\documentclass[a4paper,10pt,reqno]{amsart}

\usepackage[UKenglish]{babel}
\usepackage[utf8]{inputenc}
\usepackage[T1]{fontenc}
\usepackage{amsmath}
\usepackage{amssymb}
\usepackage{amsthm}
\usepackage{verbatim}
\usepackage{stackrel}
\usepackage[arrow, matrix, curve]{xy} 

\usepackage{orcidlink}

\usepackage[shortlabels]{enumitem}
\usepackage[nospace,noadjust]{cite}
\usepackage{comment}	
\usepackage{hyperref}
\hypersetup{
    colorlinks=true,
    linkcolor=blue,
    citecolor=red,
    filecolor=magenta,      
    urlcolor=cyan,
    pdftitle={Stability of pH-systems},
    bookmarks=true,
    linktocpage=true,
}

\usepackage{amssymb}
\usepackage{amsmath}
\usepackage{amsthm}
\usepackage{mathtools}
\usepackage{bbm}
\usepackage{verbatim, stackrel}

\usepackage[shortlabels]{enumitem}
\usepackage{marginnote}
\usepackage{color}

\usepackage{stmaryrd}
\usepackage{tikz}
\usepackage{tikz-cd}
\usepackage{tkz-graph}

\newcommand{\bbN}{\mathbb{N}}

\newcommand{\bbR}{\mathbb{R}}

\newcommand{\bbZ}{\mathbb{Z}}

\newcommand{\calA}{\mathcal{A}}

\newcommand{\calH}{\mathcal{H}}

\newcommand{\calR}{\mathcal{R}}

\newcommand{\calT}{\mathcal{T}}

\renewcommand{\i}{\mathrm{i}}
\newcommand{\e}{\mathrm{e}}
\DeclareMathOperator{\one}{{\mathbbm{1}}} 
\newcommand{\argument}{\mathord{\,\cdot\,}} 
\newcommand{\dx}{\;\mathrm{d}} 
\newcommand{\norm}[1]{\left\lVert #1 \right\rVert} 
\newcommand{\modulus}[1]{\left\lvert #1 \right\rvert} 
\newcommand{\duality}[2]{\left\langle#1\, ,\, #2\right\rangle} 
\newcommand{\dom}[1]{\operatorname{dom}\left(#1\right)} 
\DeclareMathOperator{\ran}{range} 

\newcommand{\resSet}{\rho}
\newcommand{\Res}{\mathcal{R}} 

\theoremstyle{definition}
\newtheorem{definition}{Definition}[section]
\newtheorem{remark}[definition]{Remark}

\newtheorem*{remark*}{Remark}
\newtheorem*{remarks*}{Remarks}
\newtheorem{example}[definition]{Example}

\theoremstyle{plain}
\newtheorem{proposition}[definition]{Proposition}
\newtheorem{lemma}[definition]{Lemma}
\newtheorem{theorem}[definition]{Theorem}

\numberwithin{equation}{section} 

\DeclarePairedDelimiter{\ceil}{\lceil}{\rceil}
\DeclarePairedDelimiter{\floor}{\lfloor}{\rfloor}
\newcommand{\eps}{\varepsilon}
\newcommand{\R}{\mathbb{R}}
\newcommand{\Z}{\mathbb{Z}}
\newcommand{\N}{\mathbb{N}}
\newcommand{\odd}{\text{odd}}
\newcommand{\even}{\text{even}}

\begin{document}

\title[An example for semi-uniform stability]{A universal example for quantitative semi-uniform stability}

\author[Arora]{Sahiba Arora \orcidlink{0000-0003-1973-8358}}
\address{Sahiba Arora, Department of Applied Mathematics, University of Twente, 217, 7500 AE, Enschede, The Netherlands.
Current address: Leibniz Universität Hannover, Institut für Analysis, Welfengarten 1, 30167 Hannover, Germany}
\email{sahiba.arora@math.uni-hannover.de}

\author[Schwenninger]{Felix L.~Schwenninger \orcidlink{0000-0002-2030-6504}}
\address{Felix L. Schwenninger, Department of Applied Mathematics, University of Twente, 217, 7500 AE, Enschede, The Netherlands}

\email{f.l.schwenninger@utwente.nl}

\author[Vukusic]{Ingrid Vukusic \orcidlink{0000-0002-2499-3401}}
\address{Ingrid Vukusic, Mathematics Department, University of Salzburg, Hellbrunnerstrasse 34, 5020 Salzburg, Austria.
Current address:  Department of Mathematics,
University of York,
York, North Yorkshire YO10 5GH,
United Kingdom.}
\email{ingrid.vukusic@york.ac.uk}

\author[Waurick]{Marcus Waurick \orcidlink{0000-0003-4498-3574}}
\address{Marcus Waurick, Institut für Angewandte Analysis, Technische Universität Bergakademie Freiberg, Prüferstraße 9, 09599 Freiberg, Germany}
\email{marcus.waurick@math.tu-freiberg.de}
\date{\today}
\begin{abstract}
    We characterise quantitative semi-uniform stability for $C_0$-semi\-groups arising from port-Hamiltonian systems, complementing recent works on exponential and strong stability. With the result, we present a simple universal example class of port-Hamiltonian $C_0$-semigroups exhibiting arbitrary decay rates slower than $t^{-1/2}$.
    The latter is based on results from the theory of Diophantine approximation as the decay rates will be strongly related to approximation properties of irrational numbers by rationals given through cut-offs of continued fraction expansions.
\end{abstract}

\maketitle

\section{Introduction}

Over the past fifty years operator semigroups have been established as an important framework in the context of evolution equations. While the theory may seem rather classical and largely developed, Borichev--Tomilov's seminal work \cite{BorichevTomilov2010} on \emph{polynomial stability} from 2010, inspired by an earlier result by Batty--Duyckaerts \cite{BattyDuyckaerts2008}, initiated a tremendously active area of research. The interest in \emph{semi-uniform stability}, which covers polynomial stability, arises from sharply quantifying energy-decay along classical solutions of linear partial differential equations (PDEs), where the stronger uniform exponential stability fails, but strong stability holds. Previously, this had been quantified on an ad-hoc basis, e.g.\ by exploiting the spectral decomposition of the involved operators \cite{zuazuazhang}.
Particular interest in such abstract operator-theoretic tools arises in wave equations, where different types of damping -- for instance depending on the geometry -- heavily influence asymptotic stability and accurate decay rates are desirable, see \cite{ChillSeifertTomilov2020} and the references therein. 

The success of Borichev--Tomilov's result and its consequences also lies in the abstract framework, being characteristic of the power of operator theory, from which explicit rates arise naturally from estimates over the resolvent operator, encoding hard analysis from PDEs. 
The theory of \emph{port-Hamiltonian systems} has a similar history. It developed from the paradigm that the energy flow should guide the modelling of dynamical systems, typically governed through Hamilton's principle, thereby formalising the engineering formalism represented by \emph{bond graphs}. 
In the last twenty years, 
port-Hamiltonian systems have emerged as a cornerstone framework for modelling and analysing physical phenomena, particularly those that conserve or dissipate energy. Since van der Schaft's seminal work in \cite{vanderSchaft2006}, this framework has sparked a flurry of research and development across various domains, see \cite{JacobZwart2018} for a survey. By describing physical systems through partial differential equations and boundary conditions that define ``ports'', port-Hamiltonian systems have become exemplary boundary control systems, especially for hyperbolic PDEs. We also note the close relation of port-Hamiltonian systems to hyperbolic balance laws for one spatial dimension.

The starting point of this article is the question of \emph{ ``How can we quantify decay rates of semi-uniformly stable port-Hamiltonian systems?''}. More precisely, we focus on the seemingly restrictive case of port-Hamiltonian systems with a one-dimensional spatial variable in the spirit of \cite{JacobZwart2012}. However, subtleties in assessing sharp decay rates are already present in such one-dimensional domains and are well documented in the literature \cite{ammaritucsnak,RzepnickiSchnaubelt2018,zuazuazhang}. Here, we fully characterise semi-uniform stability for this setting  and even show that the class allows for rather arbitrary decay rates.

In particular, we revisit the stability problem for port-Hamiltonian systems, focusing on providing a quantified estimate of the resolvent growth function, which, as mentioned above,  plays a significant role in characterising semi-uniform stability, see \cite{ChillSeifertTomilov2020} for a survey.  Our primary objective is to derive an estimate in terms of the matrix norm of a suitable inverse of a derived quantity, that has previously played a role in characterising exponential and strong stability of port-Hamiltonian systems in \cite{TrostorffWaurick2023} and \cite{WaurickZwart2023} respectively. This term encodes all (matrix) parameters which determine a port-Hamiltonian system in one spatial dimension. Such conditions also highlight the advantages of the (1-D-)port-Hamiltionian framework as this allows to reduce properties of a PDE to mere estimates on the algebraic building blocks. The structural insight developed in this first part then leads to the construction of a simple example showcasing algebraic decay with sharp decay rate estimates. To the best of our knowledge, this is indeed the simplest example admitting more complex decay behaviour than exponential available in the class of time-dependent partial differential equations and, thus, remarkable on its own.

Throughout this manuscript, a function $\psi$ is called \emph{increasing}, if $s\leq t$ implies $\psi(s)\leq \psi(t)$; $\psi$ is \emph{decreasing}, if $-\psi$ is increasing. 
In what follows, for an increasing function $\gamma:\bbR_+ \to (0,\infty)$, we write $\gamma^{-1}$ for the right-continuous \emph{generalized inverse} of $\gamma$ which is given by
\begin{align*}
    \gamma^{-1}(y)\coloneqq \sup\{ x\ge 0 : \gamma(x)\le y\}.
\end{align*}
We refer to \cite{Wacker2023} for properties of generalized inverses. 

\subsection*{Organisation of the article}

The paper is organised as follows:~In Section~\ref{sec:port-hamiltonian}, we discuss  port-Hamiltonian systems and present the first main result characterising their polynomial stability. In particular, we obtain decay rates for our semigroup entirely in terms of the matrix norm of a suitable inverse associated with the corresponding port-Hamiltonian system. 

In Section~\ref{sec:exaeasy} we use the characterization from Section~\ref{sec:port-hamiltonian} to show, for instance, that for (almost) any decay rate slower than $t^{-1/2}$ there is an example in  the class of one-dimensional port-Hamiltonian systems admitting this decay rate. For this, we rely on known operator-theoretic results, that is, quantitative decay rates for $C_0$-semigroups, as well as results from Diophantine approximation. These are collected in Appendices~\ref{appendix:stability-background} and \ref{appendix:diophantine} respectively, since the precise formulations of these  results are instrumental for the proofs of our decay rates.
Note that the necessity of Diophantine approximations in this context does not appear to be a coincidence, but rather a general feature of the (algebraic) stability analysis of $C_0$-semigroups. Indeed, connections to Diophantine approximations,  dynamical systems and number theory have been observed in the context of decay of orbits of solutions of time-dependent PDEs, for instance, in \cite[Section~3]{RzepnickiSchnaubelt2018}, \cite[Section~6B]{CPSST23}, and \cite[Sections~4 and 5]{WaurickZwart2023}. 

Finally, Sections~\ref{sec:proof-square},~\ref{sec:proof-othertypes}, and~\ref{sec:posinc} contain the proofs of the number-theoretic results stated and used in Section~\ref{sec:exaeasy}. For these proofs, we require a means to quantify the error for odd/odd rational approximations of irrational number~(Section~\ref{sec:ooappr}).

\section{Resolvent norms on $i\mathbb{R}$ for Port-Hamiltonian Operators}
    \label{sec:port-hamiltonian}

Fix $d\in \bbN$. Let $\calH:[a,b]\to \bbR^{d\times d}$ be a measurable function with $\calH(t)^\top=\calH(t)$ for a.e.\ $t\in[a,b]$, and which is assumed to be bounded above and strictly bounded away from $0$ in the sense of positive definiteness. The Hilbert space $H\coloneqq L^2([a,b]; \bbR^d)$ is equipped with the norm
\[
    \norm{\argument}_H : = \norm{\calH^{1/2}\argument}_{L^2}.
\]
Let $P_0, P_1 \in \bbR^{d\times d}$ be such that $P_0^\top=-P_0$, $P_1$ is self-adjoint and invertible. In addition, let $W\in\R^{d\times 2d}$ have full rank.

On $H$ we define the operator $\calA$ by
\begin{align}
    \label{eq:original-gernerator}
    \begin{split}
        \dom{\calA} &= \left\{ u \in H: \calH u \in H^1([a,b];\bbR^d), W \begin{bmatrix} (\calH u)(b) \\ (\calH u)(a) \end{bmatrix} = 0 \right\}\\
    \calA u    &= P_1(\calH u)' +P_0 \calH u.
    \end{split}
\end{align}
We further assume that $\calA$ is dissipative, i.e., $\Re\langle \calA x,x\rangle\leq 0$ for all $x\in \dom{\calA}$, which can equivalently be rephrased by a matrix inequality involving $P_{1}$ and $W$, see e.g.\  \cite[Theorem~2.4]{TrostorffWaurick2023}.
Consequently, $-\calA$ generates a strongly continuous semigroup of contractions $\calT$ on $H$. Moreover, $\calA$ has compact resolvent and hence, its spectrum consists only of eigenvalues (of finite multiplicity)  \cite[Theorem~2.2 and Lemma~3.3]{WaurickZwart2023}. 
Under these assumptions, we will loosely refer to the collection of (matrix) parameters $\mathcal{H}, P_{0},P_{1}, W$ and the associated operator $\calA$ as \emph{port-Hamiltonian system}, see e.g.\ \cite{ZwartMaschke2010,JacobZwart2012,JacobZwart2018} for the origin of the terminology within control theory. These systems are sometimes referred to as \emph{1-D hyberbolic balance laws}, e.g.\ \cite{BastinCoron2016}.

The space $H$ and its norm naturally arise in modeling systems through energy balances and flows, with port-Hamiltonian operators $\calA$ being prototypical for (linear) hyperbolic partial differential equations, such as appearing in transport or vibration phenomena or fluid dynamics.

Here and in the sequel, $\one_d$ denotes the identity matrix of dimension $d$. For $t\in\R$, the solution $v=\Phi_t:[a,b]\to\R^{d\times d}$ associated to the matrix-valued ODE-system
    \begin{align*}
             v'(x)  & = -P_1^{-1}(\i t\calH(x)^{-1}+P_0) v(x) \quad \text{for }  x\in [a,b],  \\
             v(a)   & = \one_d                                     
    \end{align*}
is referred to as fundamental matrix (of the port-Hamiltonian system).
The connection between the fundamental matrix and the resolvent $\calR(\cdot,A)$ (along the imaginary axis)   is reflected in the following lemma, which already implicitly appeared in the proof of \cite[Theorem~3.4]{TrostorffWaurick2023}. We point out that~\eqref{eq:resolvent-formula} below is simply the variation of parameters formula associated to $\calA$.

\begin{lemma}
    \label{lem:resolvent-formula}
    Let $t\in \bbR$ and $u,f\in H$. Then $u=\Res(\i t,-\calA)f$ if and only if
    \begin{equation}
        \label{eq:resolvent-formula}
        (\calH u)(x) = \Phi_t(x)  (\calH u)(a) + \Phi_t(x) \int_a^x \Phi_t(s)^{-1}P_1^{-1}f(s)\dx s
    \end{equation}
    and
    \begin{equation}
        \label{eq:resolvent-condition}
        W
        \begin{bmatrix}
            \Phi_t(b)\\
            \one_d
        \end{bmatrix}
        (\calH u)(a) 
        = - W
        \begin{bmatrix}
            \Phi_t(b) \int_a^b \Phi_t(s)^{-1}P_1^{-1}f(s)\dx s\\
            0
        \end{bmatrix}.
    \end{equation}
\end{lemma}

\begin{proof}
    On the space $L^2([a,b];\bbR^d)$, consider the operator
    \begin{align*}
        \dom{A} &\coloneqq \left\{ v \in H^1([a,b],\bbR^d):  W \begin{bmatrix} v(b) \\ v(a) \end{bmatrix} = 0 \right\}\\
        A v    &\coloneqq P_1 v' +P_0 v.
    \end{align*}
    Observe that $(\i t+\calA)u =f$ if and only if $\calH u \in \dom A$ and 
    \[
        (\calH u)'= - P_1^{-1}(\i t\calH^{-1} +P_0)\calH u + P_1^{-1}f.
    \]
    By definition of the fundamental matrix, the latter is equivalent to~\eqref{eq:resolvent-formula}. Whence,  
    \begin{align*}
         \begin{bmatrix} (\calH u)(b) \\ (\calH u)(a) \end{bmatrix}  
           =
        \begin{bmatrix}
            \Phi_t(b)\\
            \one_d
        \end{bmatrix}(\calH u)(a)  
        +
        \begin{bmatrix}
            \Phi_t(b) \int_a^b \Phi_t(s)^{-1}P_1^{-1}f(s)\dx s\\
            0
        \end{bmatrix}.
    \end{align*}
    In other words, $\calH u \in \dom A$ is equivalent to~\eqref{eq:resolvent-condition}.
\end{proof}

We recall the main result from \cite{WaurickZwart2023} characterising strong stability of $\calT$.
 As usual $\resSet(A)$ denotes the resolvent set of $A$.
\begin{theorem}[{{\cite[Theorem~1.3]{WaurickZwart2023}}}]\label{thm:WZmain}
Given the setting in the present section, the following conditions are equivalent.
\begin{enumerate}
    \item  The semigroup $\calT$ is strongly stable, i.e., for each $x\in X$, $\calT(t)x\to 0$ as $t\to\infty$.
    \item There are no spectral values of $\calA$ on the imaginary axis,  i.e., $\i\bbR \subseteq\resSet(\calA)$.
    \item  The matrix $
    T_t\coloneqq W
    \begin{bmatrix}
        \Phi_t(b)\\
        \one_d
    \end{bmatrix}
    $
    is invertible for each $t\in \bbR$.
\end{enumerate}
\end{theorem}

Recall that a strongly continuous semigroup $\calT$ with generator $\calA$ is called \emph{semi-uniformly stable}, if \[\lim_{t\to\infty}\norm{\calT(t)\calA^{-1}}=0.\] Our aim in this article is to give precise estimates on the function
\[
    M(\eta)\coloneqq \sup_{t\in [-\eta,\eta]} \norm{\Res(\i t,-\calA)}
\]
in order to obtain sharp decay rates on $\norm{\calT(t)\calA^{-1}}$ as $t\to\infty$ by celebrated results on semi-uniform stability, see Appendix~\ref{appendix:stability-background}.
Essentially, we show that the fundamental matrix $\Phi_{t}$ and $W$
determine the leading order of the resolvent norm on the imaginary axis and hence control this function.

For the rest of the paper, all matrix norms refer to the ones induced by the Euclidean norm.

\begin{theorem}\label{thm:MforPhs}
     Let $-\calA$ be the generator of the contraction semigroup $\calT$ above and assume that $\i \R\subseteq \resSet(\calA)$ and $\sup_{t\in \bbR}\norm{\Phi_t}_{\infty} <\infty$. 
     For any $t\in \bbR$, 
        \[T_t\coloneqq W
        \begin{bmatrix}
            \Phi_t(b)\\
            \one_d
       \end{bmatrix}\in\R^{d\times d}\] is invertible, and there exist $c,C>0$ such that
   \[
          c m(\eta)\leq M(\eta)\leq C m(\eta) \quad \text{for all } \eta\ge 0;
    \]
    where
    \[
         m(\eta)\coloneqq \sup_{t\in [-\eta,\eta]} \norm{T_t^{-1}}.
     \]
\end{theorem}

The assumption $\sup_{t\in \bbR}\norm{\Phi_t} _{\infty} <\infty$ already appeared in \cite{TrostorffWaurick2023} in order to obtain a characterisation for exponential stability of port-Hamiltonian systems. In particular, it is satisfied if $\mathcal{H}$ is of bounded variation 
or if $\mathcal{H}(x)E_+\subseteq E_+$ for a.e.~$x\in (a,b)$ with $E_+$ being the spectral subspace of positive eigenvalues of $P_1$.
For these results and further discussion on the condition we refer the reader to \cite[Section~6]{TrostorffWaurick2023}.

We recall that a continuous function $\gamma:\bbR_+\to (0,\infty)$ is said to be of \emph{positive increase} if there are $\alpha, t_0>0$ and $c\in (0,1]$ such that
\begin{equation}\label{eq:posinc}
    \frac{ \gamma(\lambda t)}{\gamma(t)} \ge c \lambda^{\alpha} \quad \text{whenever }\lambda\ge 1 \text{ and }t\ge t_0.
\end{equation}
In particular, polynomials and \emph{regularly varying} functions of positive index have a positive increase; see \cite{BattyChillTomilov2016,RozendaalSeifertStahn2019}.

\begin{remark}\label{rem:positive_increase_sandwich}
In Theorem~\ref{thm:MforPhs}, if the function $m$ is of positive increase, then so is the function $M$ associated to the resolvent. Indeed,
let $\gamma$ be of positive increase as in~\eqref{eq:posinc} and let $\delta \colon (0,\infty)\to(0,\infty)$ satisfy
  \[
     d \gamma(t)\leq \delta(t) \leq D \gamma(t)\qquad(t\geq t_0),
  \]
  for some $d,D,t_0>0$. Then letting $\alpha>0$, $c\in (0,1]$ according to~\eqref{eq:posinc} yields for $t\geq t_0$ and $\lambda\geq 1$ that
  \[
    \frac{\delta(\lambda t)}{\delta(t)}\geq  \frac{d \gamma(\lambda t)}{D \gamma(t)}\ge \frac{d}{D} c \lambda^\alpha.
  \]
  In other words, $\delta$ is of positive increase as well.
\end{remark}

We split the proof of Theorem~\ref{thm:MforPhs} across the following two lemmata. Observe the topological isomorphism of Hilbert spaces
$S\colon u\mapsto \calH^{-1}u$ from $L^2([a,b];\bbR^d)$ to $H$; see e.g.\ \cite[Lemma~3.1]{TrostorffWaurick2023}.
\begin{lemma}
    \label{lem:sufficient}
    Assume that $\i\bbR \subseteq\resSet(\calA)$ and that $B\coloneqq\sup_{t\in \bbR}\norm{\Phi_t}_{\infty} <\infty$. Then for each $t\in \bbR$, the matrix
    $
    T_t\coloneqq W
    \begin{bmatrix}
        \Phi_t(b)\\
        \one_d
    \end{bmatrix}
    $
    is invertible and
    \begin{equation}
        \label{eq:resolvent-estimate-upper}
        \norm{ \Res(\i t, -\calA)} \leq \widetilde C( \norm{T_t^{-1}}+1);
    \end{equation}
    where
    \[
        \widetilde C\coloneqq(b-a)B^2\norm{S}\norm{P_1} \norm{P_1^{-1}}^2 \max\{B\norm{W},(b-a)^{1/2}\}.
    \]
\end{lemma}

\begin{proof}
     First of all, recall that $\norm{\Phi_t^{-1}}_{\infty}\leq  B\norm{P_1}  \norm{P_1^{-1}}$ from \cite[Lemma~3.2]{TrostorffWaurick2023} and from Theorem~\ref{thm:WZmain} that $\i\bbR\subseteq \resSet(\calA)$ implies the invertibility of $T_t$ for all $t\in \R$.
     
    Let $f\in H$ and $t\in \bbR$.
    Setting $v\coloneqq S^{-1}\Res(\i t,-\calA)f$, we obtain from~\eqref{eq:resolvent-formula} that $\norm{\Res(\i t,-\calA)f}$ can be estimated from above by
    \begin{align*}
         \norm{S} \norm{v}_{L^2} 
                                  \le (b-a)^{1/2} \norm{\Phi_t}_{\infty}\norm{S} \left( \norm{v(a)} + (b-a) \norm{\Phi_t}_{\infty} \norm{P_1} \norm{P_1^{-1}}^2 \norm{f}_{L^2}\right).
    \end{align*}
    On the other hand,~\eqref{eq:resolvent-condition} and invertibility of $T_t$ together imply that 
    \begin{align}
        \label{eq:norm-at-a}
        \begin{split} 
            \norm{v(a)} & = \norm{T_t^{-1}W  
            \begin{bmatrix}
                \Phi_t(b) \int_a^b \Phi_t(s)^{-1}P_1^{-1}f(s)\dx s\\
                0
            \end{bmatrix} }\\
            &\le  (b-a)^{1/2} \norm{\Phi_t}_{\infty}^2\norm{W}\norm{P_1}  \norm{P_1^{-1}}^2 \norm{T_t^{-1}}  \norm{f}_{L^2}.
        \end{split}
    \end{align}
    It follows that
    \[
        \norm{ \Res(\i t, -\calA)} \le (b-a)B^2\norm{S}\norm{P_1} \norm{P_1^{-1}}^2 \left(  B\norm{W}  \norm{T_t^{-1}}   +  (b-a)^{1/2} \right)
    \]
    and hence the assertion.
\end{proof}

\begin{lemma}
    \label{prop:necessary}
   Let $t\in \bbR$. If 
   $
    T_t\coloneqq W
    \begin{bmatrix}
        \Phi_t(b)\\
        \one_d
    \end{bmatrix}
    $
    is invertible, then
    \[
        \norm{T_t^{-1}}\le C_t\left( (b-a)^{1/2} \|\calH\|_\infty\norm{P_1} \norm{\Res(\i t,-\calA)} +1\right) ;
    \]
    where 
    \[
        C_t\coloneqq \left( \frac{1}{(b-a)^{3/2}} \norm{P_1^{-1}}^2 \norm{P_1}^2 \norm{\Phi_t}_{\infty}^3 \left(1+\norm{\Phi_t}_{\infty}\right)+1\right) \norm{W^+};
    \]
    here $W^+$ denotes the Moore--Penrose inverse of $W$.
\end{lemma}
\begin{proof}
    By compactness, we find $z\in \bbR^d$ with $\norm{z}=1$ such that 
    \[
        \norm{T_t^{-1}z} = \norm{T_t^{-1}}.
    \]
    Since $W$ has full rank, the element  
    $\begin{bmatrix}
         z_1 &  z_2
    \end{bmatrix}^\top\coloneqq W^{+}z$ satisfies
    \[
        z = W 
        \begin{bmatrix}
             z_1 \\  z_2
        \end{bmatrix}.
    \]
    Next, we set
    \[
        y\coloneqq -z_1+ \Phi_t(b)z_2 \quad \text{ and }\quad f\coloneqq (b-a)^{-1} P_1 \Phi_t(\argument)\Phi_t(b)^{-1}y.
    \]
    Of course,
    $
        \norm{y} \le \left(1+ \norm{\Phi_t}_{\infty}\right)\norm{W^+ z} 
    $
    and furthermore, by \cite[Lemma~3.2]{TrostorffWaurick2023} we have
    $f\in L^\infty([a,b],\bbR^d)$ with 
    \begin{align}
        \begin{split}
            \label{eq:sup-norm-f}
           \norm{f}_{\infty} &\le (b-a)^{-1} \norm{P_1} \norm{\Phi_t}_{\infty} \big(\norm{P_1^{-1}} \norm{P_1} \norm{\Phi_t(b)}\big) \norm{y}\\
                            & \le (b-a)^{-1} \norm{\Phi_t}_{\infty}^2(1+\norm{\Phi_t}_{\infty}) \norm{P_1}^2 \norm{P_1^{-1}}\norm{W^{+}z}.
        \end{split}
    \end{align}
    Next,~\eqref{eq:resolvent-condition} along with our choice of $f$ yields that $u\coloneqq \Res(\i t,-\calA)f$ satisfies
    \begin{align*}
        T_t(\calH u)(a)&= -W 
                    \begin{bmatrix}
                        y \\
                        0  
                    \end{bmatrix}
                    = W 
                    \begin{bmatrix}
                        z_1- \Phi_t(b)z_2 \\ 
                        0
                    \end{bmatrix}\\
                    & = W
                    \begin{bmatrix}
                        z_1 \\
                        z_2
                    \end{bmatrix}
                    - W
                    \begin{bmatrix}
                        \Phi_t(b)\\
                        \one_d
                    \end{bmatrix}z_2\\
                    & = z-T_t z_2.
    \end{align*}
    Invertibility of $T_t$ thus gives $T_t^{-1}z=(\calH u)(a)+z_2$ and whence
    \begin{equation}
        \label{eq:inverse-estimate-preliminary}
       \norm{T_t^{-1}}=\norm{T_t^{-1}z} \le \norm{(\calH u)(a)}+\norm{z_2} \le \norm{(\calH u)(a)}+\norm{W^+z}.
    \end{equation}

    It remains to estimate $\norm{(\calH u)(a)}$ for which we recall from Lemma~\ref{lem:resolvent-formula} that
    \[
        (\calH u) (a) = \Phi_t(x)^{-1}(\calH u)(x)- \int_a^x \Phi_t(s)P_1^{-1}f(s)\dx s.
    \]
    Therefore, we can once again use \cite[Lemma~3.2]{TrostorffWaurick2023} to deduce that
    \begin{align*}
        \norm{\Phi_t}_{\infty}^{-1}\norm{P_1^{-1}}^{-1}\norm{(\calH u)(a)  }_2 
                              & =      \norm{P_1}  \norm{\calH u}_{L^2} +  \norm{f}_{\infty}\\
                              & =  \norm{P_1}  \norm{\calH \Res(\i t,-\calA)f}_{L^2} +  \norm{f}_{\infty}\\
                              & \le  \norm{\calH}_\infty \norm{P_1}  \norm{ \Res(\i t,-\calA)}\norm{f}_{L^2} +  \norm{f}_{\infty}\\
                              & \le \left((b-a)^{1/2}\|\calH\|_\infty \norm{P_1}  \norm{ \Res(\i t,-\calA)} +  1\right)\norm{f}_{\infty}.
    \end{align*}
    In particular, we infer from~\eqref{eq:sup-norm-f} that
    \begin{align*}
        \norm{(\calH u)(a)  } &\le (b-a)^{-1/2} \norm{(\calH u)(a)  }_2 \\
         &\le \left(C_t \norm{W^+}^{-1}-1\right)
          \left((b-a)^{1/2}\|\calH\|_\infty \norm{P_1}  \norm{ \Res(\i t,-\calA)} +  1\right)\norm{W^{+}z}.
    \end{align*}
    Substituting into~\eqref{eq:inverse-estimate-preliminary} and using $\norm{z}=1$, we deduce the assertion.
\end{proof}

\begin{proof}[Proof of Theorem~\ref{thm:MforPhs}]
    From Lemma~\ref{lem:sufficient} -- in particular, the estimate~\eqref{eq:resolvent-estimate-upper} -- we obtain the upper bound for the function $M$ associated to the resolvent. 
    On the other hand, Lemma~\ref{prop:necessary} yields that
    \[  
        \norm{T_t^{-1}} \le C  \left(  \norm{\Res(\i t,-\calA)} +1\right);
    \]
    where the constant $C$ is given by
    \[
        \left( \frac{1}{(b-a)^{3/2}}  \norm{P_1}^2 \norm{P_1^{-1}}^2 B^3 (1+B)+1\right)\norm{W^+} \max\left\{(b-a)^{1/2} \|\calH\|_\infty \norm{P_1}, 1\right\}.
    \]
\end{proof}
  
\section{A universal example for stability types}
    \label{sec:exaeasy}

    The main example we focus on in the following is a generalisation of  \cite[Example~4.3]{TrostorffWaurick2023}, where, using the notation below,  $\alpha=\sqrt{2}$ was taken. For the general case considered here, let $\alpha \in (0,\infty)$ and note that indeed the particular choice of $\alpha$ will crucially influence the decay rates. On $L^2([0,1];\bbR^2)$ consider the port-Hamiltonian system given by
\begin{equation}
    \label{eq:universal-example}
    \calH\coloneqq 
    \begin{bmatrix}
        1 & 0\\
        0 & \alpha
    \end{bmatrix}^{-1},
    \ 
    P_1 \coloneqq \one_2=
    \begin{bmatrix}
        1 & 0\\
        0 & 1
    \end{bmatrix}, \  P_0\coloneqq 0,
    \ 
    \text{ and }
    \ 
    W\coloneqq  
    \begin{bmatrix}
        \widetilde W & \one_2
    \end{bmatrix}
\end{equation}
in~\eqref{eq:original-gernerator};
where 
$
    \widetilde W\coloneqq\frac12
    \begin{bmatrix}
        1 & 1 \\
        1 & 1
    \end{bmatrix}.
$
By \cite[Theorem~2.4(iv) and Example~4.3]{TrostorffWaurick2023}, the operator $-\calA$ given by~\eqref{eq:original-gernerator} generates a contraction semigroup $\calT^\alpha$ with
\begin{equation}\label{eq:Talphat}
    T_{t,\alpha}= \widetilde W\e^{\i t\calH^{-1}}+\one_2
    \quad
    \text{and}
    \quad
    \det(T_{t,\alpha})= 1+\frac12\left(\e^{\i t}+\e^{\i\alpha t}\right);
\end{equation}
here $T_{t,\alpha}$ denotes the matrix $T_t$ in Theorem~\ref{thm:MforPhs} corresponding to $\calT^\alpha$.

Note that in the present case, $\calH$ is constant, and, thus, of bounded variation, so that the condition $\sup_{t\in \bbR}\norm{\Phi_t}_{\infty} <\infty$ on the fundamental matrices is satisfied.
It was also shown in \cite[Example~4.3]{TrostorffWaurick2023} that $\calT^\alpha$ is not exponentially stable for $\alpha=\sqrt{2}$. Actually, even more is true:

\begin{proposition}\label{prop:strongstab} 
    The semigroup $\calT^\alpha$ is not exponentially stable for any $\alpha \in (0,\infty)$. Moreover,
    the following conditions are equivalent:
    \begin{enumerate}
      \item The semigroup $\calT^\alpha$ is strongly stable.
      \item The number $\alpha$ is irrational.
    \end{enumerate} 
\end{proposition}

\begin{proof}
     The equivalence of strong stability and irrationality of $\alpha$ follows by a direct application of Theorem~\ref{thm:WZmain} keeping in mind the formula of $\det(T_{t,\alpha})$ given in~\eqref{eq:Talphat}.
     
     On the other hand, we know from \cite[Theorem~3.5]{TrostorffWaurick2023}, that $\calT^\alpha$ is exponentially stable if and only if the function $\eta\mapsto m_{\alpha}(\eta) = \sup_{t\in [-\eta,\eta]} \norm{T_{t,\alpha}^{-1}}$ is uniformly bounded. Now, for irrational $\alpha$, we have from Proposition~\ref{prop:square} below that $m_{\alpha}$ is unbounded, whence the semigroup can't be exponentially stable.
\end{proof}

The following two theorems -- in which we examine the system associated with~\eqref{eq:original-gernerator} for non-exponential stability -- are the main results of this article. 
Recall that a number $\alpha\in (0,\infty)$ is called \emph{badly approximable} if there exists $c>0$ such that for all $p,q\in \Z$ with $q\neq 0$, we have
\[
	\left| \alpha - \frac{p}{q} \right|
	\geq \frac{c}{q^2}.
\]

\begin{theorem}\label{thm:decayest} 
    For the port-Hamiltonian system~\eqref{eq:universal-example}, the following bounds hold.
    \begin{enumerate}[\upshape (a)]
        \item For Lebesgue a.e.~$\alpha \in (0,\infty)$ and all $\varepsilon>0$, there exist $C,t_0>0$ with
        \[
            \norm{\calT^\alpha(t)\calA^{-1}}\leq%
            \frac{C(\log t)^{1+\epsilon}}{t^{1/2}}%
            \quad \text{for all }t\ge t_0.
        \]
      \item If $\alpha\in (0,\infty)$ is badly approximable, then there exist $c,C,t_0>0$ such that
        \[
           \frac{c}{\sqrt{t}}\leq  \norm{\calT^\alpha(t)\calA^{-1}}\leq \frac{C}{\sqrt{t}}\quad \text{for all }t\ge t_0.
        \]
        \item If $\gamma \colon (0,\infty)\to (0,\infty)$ is increasing with 
        $\lim_{t\to \infty} t^{-2}\gamma(t)=\infty$, then there exist $\alpha\in (0,\infty)$ and $c,C, t_0>0$ such that
        \[
            \norm{\calT^\alpha(t)\calA^{-1}}\leq \frac{c}{\gamma_{\log}^{-1}(t/C)}\quad \text{for all }t\ge t_0;
        \]
        where
        \[
            \gamma_{\log}(\eta) \coloneqq \gamma(\eta)\big(\log(1+\gamma(\eta))+\log(1+\eta)\big)\qquad (\eta> 0).
        \]
        In addition, if $\gamma$ is continuous and of positive increase as in~\eqref{eq:posinc}, then the latter estimate can be replaced by
        \[
           \norm{\calT^\alpha(t)\calA^{-1}}\leq \frac{c}{\gamma^{-1}(t)}
           \quad \text{for all }t\ge t_0
        \]
        for some $c>0$.
    \end{enumerate}
\end{theorem}

The stability rates in Theorem~\ref{thm:decayest}(b) are optimal, while,  in Theorem~\ref{thm:decayest}(a) up to a logarithmic factor, see the first statement in the following theorem. The second statement shows in which sense the rates in Theorem~\ref{thm:decayest}(c) are optimal.

\begin{theorem}\label{thm:decayopt}
    For the port-Hamiltonian system~\eqref{eq:universal-example}, the following hold.
    \begin{enumerate}[\upshape (a)]
    \item For all irrational $\alpha \in (0,\infty)$, 
    there exist $c>0$  such that for each $t_0>0$, there exists $t\ge t_0$ with  $\sqrt{t}\norm{\calT^\alpha(t)\calA^{-1}}\ge c$;
    equivalently,
    \[
        \limsup_{t\to\infty}\sqrt{t}\norm{\calT^\alpha(t)\calA^{-1}}>0.
    \]
    
    \item If $\gamma \colon (0,\infty)\to (0,\infty)$ is %
    continuous, increasing, and unbounded,
     then there exist $\alpha\in (0,\infty)$ and $c>0$ such that
    \[
       \limsup_{t\to\infty}\gamma^{-1}(ct) \norm{\calT^\alpha(t)\calA^{-1}}>0.
    \] 
    Furthermore, if $\gamma$ is of positive increase, any $c>0$ may be chosen.
    \end{enumerate}
\end{theorem}

We denote the function $m$ in Theorem~\ref{thm:MforPhs} corresponding to $\calT^\alpha$ by $m_\alpha$ and the function $M$ associated to the resolvent by $M_{\alpha}$. The proofs of Theorems~\ref{thm:decayest} and~\ref{thm:decayopt} -- given at the end of this section -- rely on the properties of  $m_{\alpha}$, that are collected in the following propositions.

\begin{proposition}\label{prop:square} 
    Let $\varepsilon>0$. Then for Lebesgue a.e.~$\alpha \in (0,\infty)$ there exist $C,\eta_0>0$ such that for all $\eta\in [\eta_0,\infty)$, the function $m_{\alpha}(\eta)\coloneqq \sup_{t\in [-\eta,\eta]} \norm{T_{t,\alpha}^{-1}}$ satisfies
    \[
        m_\alpha (\eta) \leq C \eta^2 (\log \eta)^{2+\varepsilon}.
    \]
    Moreover, for all irrational $\alpha\in (0,\infty)$ there is $c>0$ such that for all $t_0>0$ there exists $t\geq t_0$ with
    \[
       c t^2 \leq \norm{T_{t,\alpha}^{-1}}.
    \]
\end{proposition}
Even though the set of badly approximable numbers has Lebesgue measure zero \cite[Theorem 29]{Kh63}, its Hausdorff dimension is $1$ \cite{Jarnik29}.
In particular, $m_{\alpha}$ has positive increase precisely on a set that is “large” in terms of dimension but “small” in terms of measure, as formalised in the following result.

\begin{proposition}\label{prop:badsquare} 
    For $\alpha \in (0,\infty)$, let $m_{\alpha}(\eta)\coloneqq \sup_{t\in [-\eta,\eta]} \norm{T_{t,\alpha}^{-1}}$.
    \begin{enumerate}[\upshape (a)]
        \item If $\alpha\in (0,\infty)$ is badly approximable, then there exists $c,C,\eta_0>0$ such that for all $\eta\geq \eta_0$, we have
        \[
           c\eta^2 \leq  m_\alpha (\eta) \leq C\eta^2;
        \]
        and hence $m_\alpha$ is of positive increase (see Remark~\ref{rem:positive_increase_sandwich}).

        \item The set
            \[
               \{ \alpha \in (0,\infty): m_\alpha \text{ is of positive increase}\}
            \]
            is a Lebesgue null set. 
    \end{enumerate}
\end{proposition}

Next, we turn to more particular constructions allowing for more specific asymptotic behaviour of the resolvent/semigroup.

\begin{proposition}\label{prop:othertypes}
    If $\gamma \colon (0,\infty)\to (0,\infty)$ is increasing, then 
    there exist $\alpha\in (0,\infty)$ and $c>0$ such that
    for all $t_0>\pi$, there exists $t\geq t_0$ with
    \[
       c\gamma(t - \pi ) \leq \norm{T_{t,\alpha}^{-1}}.
    \]
    In addition, if
    \[
        \lim_{t\to \infty} t^{-2}\gamma(t)=\infty,
    \]
    then there exist $C,\eta_0>0$ such that 
    \[
        m_\alpha (\eta) \leq C\gamma(\eta + \pi)\qquad(\eta\geq \eta_0),
    \]
    where $m_{\alpha}(\eta)\coloneqq \sup_{t\in [-\eta,\eta]} \norm{T_{t,\alpha}^{-1}}$.
\end{proposition}

\begin{remark}
    In Proposition~\ref{prop:othertypes}, even if $\gamma$ is of positive increase, $m_\alpha$ might not be. Indeed, in Example~\ref{ex:alpha_wellapprox} below, $\gamma(t)=\frac1{f(t)} =e^t$ is of positive increase. Yet $m_{\alpha}$  for the constructed $\alpha$ is not of positive increase due to Lemma~\ref{lem:not_pos_incr}; see Example~\ref{ex:odd-odd-gaps}.
\end{remark}

The proofs of Propositions~\ref{prop:square},~\ref{prop:othertypes}, and~\ref{prop:badsquare} are based on approximations of irrational numbers by rationals and are postponed to Sections~\ref{sec:proof-square},~\ref{sec:proof-othertypes}, and~\ref{sec:posinc} respectively.

\begin{proof}[Proof of Theorem~\ref{thm:decayest}]
    (a) Note that $t\mapsto \delta(t)\coloneqq t^2 (\log t)^{2+\varepsilon}$ is of positive increase. Therefore, combining Proposition~\ref{prop:square} and Theorem~\ref{thm:MforPhs} yields that the assumptions of Theorem~\ref{thm:RSS32} are fulfilled for $\gamma = \delta$. Consequently,
    we obtain for Lebesgue a.e.~$\alpha \in (0,\infty)$ and all $\varepsilon>0$, there exist $c,C,t_0>0$ with
    \[
            \norm{\calT^\alpha(t)\calA^{-1}}\leq \frac{C}{\delta^{-1}(t)}\quad \text{for all }t\ge t_0.
    \]
    
    Replacing $\epsilon$ by $2\epsilon$, it remains to show that $\delta^{-1}(t) \sim s(t)\coloneqq 2^{1+\epsilon/2}t^{1/2}(\log t)^{-1-\epsilon/2}$.
    Firstly, from $\frac{\log(\log t)}{\log t} \to 0$ as $t\to\infty$ and
    \[
        \log( s(t)) = \log(2^{1+\epsilon/2}) +\frac12\log t-\left(1+\frac{\epsilon}2\right)\log(\log t)
    \]
    we conclude that $\log( s(t)) \sim \frac12 \log t$. Thus
    \begin{align*}
        \delta(s(t))  = s(t)^2 \log( s(t) )^{2+\epsilon}
                      \sim 2^{2+\epsilon}t(\log t)^{-2-\epsilon} \left(\frac12 \log t\right)^{2+\epsilon}
                      = t,
    \end{align*}    
    i.e., $\delta^{-1}(t) \sim s(t)$.

    (b) If $\alpha$ is badly approximable, then $m_{\alpha}$ is of positive increase by Proposition~\ref{prop:badsquare}(a). In turn, Remark~\ref{rem:positive_increase_sandwich} and Theorem~\ref{thm:MforPhs} imply that $M_{\alpha}$ associated to the resolvent is of positive increase as well.
    Therefore, we can apply Theorem~\ref{thm:RSS11} to obtain the desired estimate with $M_{\alpha}^{-1}(t)$ instead of $\sqrt t$.
    Employing Proposition~\ref{prop:badsquare}(a) and Theorem~\ref{thm:MforPhs} once again yields the result; cf. Remark~\ref{rem:inverse-log-sandwich}.

    (c) 
    Firstly, it follows from Theorem~\ref{thm:MforPhs} and Proposition~\ref{prop:othertypes} that there exists $\alpha\in (0,\infty)$, $C'>0$ and $\eta_{0}>0$ such that $M_{\alpha, \log}(\eta)\le C'\gamma_{\log}(\eta+\pi)$ for all $\eta\ge\eta_{0}$. From this, it is easy to show that there exists $C>0$ and $s_{0}>0$ such that
    \begin{equation}
        \label{eq:decayest:log}
        \gamma_{\log}^{-1}\left(\frac{s}{C}\right) -\pi \le M_{\alpha,\log}^{-1}(s) \quad (s\ge s_{0}).
    \end{equation}
    Since $\gamma_{\log}^{-1}$ is increasing with $\lim_{s\to\infty}\gamma_{\log}^{-1}(s)=\infty$, the left-hand side is bounded below by $\frac{1}{2}\gamma_{\log}^{-1}(s/C)$ for all $s\ge s_{1}$ and some $s_{1}\ge s_{0}$.
    On the other hand, due to Theorem~\ref{thm:borichev-tomilov}, there exists $c, t_0>0$ such that
    \begin{equation}
        \label{eq:decayest:semigroup}
        \norm{\calT^\alpha(t)\calA^{-1}} \le \frac{c}{M_{\alpha,\log}^{-1}(t/c)} \quad \text{for all }t\ge t_0.
    \end{equation}
    
    Combining~\eqref{eq:decayest:log} and~\eqref{eq:decayest:semigroup} yields asserted estimate.%
    Furthermore, if $\gamma$ is continuous and of positive increase, then the second inequality is a consequence of Theorems~\ref{thm:RSS32} and~\ref{thm:MforPhs} and Proposition~\ref{prop:othertypes}.
\end{proof}

\begin{proof}[Proof of Theorem~\ref{thm:decayopt}]
    Let $\alpha \in (0,\infty)$ be irrational. 

    (a) We know from Proposition~\ref{prop:square} and Theorem~\ref{thm:MforPhs} that there exists $d>0$ such that for each $\eta_0>0$, there exists $\eta\ge \eta_0$ with $M_{\alpha}(\eta)\ge d \eta^2$. As $M_{\alpha}$ is unbounded, we also have from Theorem~\ref{thm:lowerbound} that there exist $c, C, t_0>0$ such that
    \[
        \norm{\calT^\alpha(t)\calA^{-1}} \ge \frac{c}{M_{\alpha}^{-1}(Ct)}
    \]
    for all $t\ge t_0$. The result now follows because $\liminf_{t\to\infty}M_{\alpha}^{-1}(t)\le \sqrt{d^{-1}t}$.
    
    (b) The assertion follows by combining Theorem~\ref{thm:MforPhs} with Propositions~\ref{prop:othertypes} and~\ref{prop:CPSST23-53}.
\end{proof}

\section{A means to quantify odd over odd rational approximations}
    \label{sec:ooappr}

For the proofs of our main results related to the example in Section~\ref{sec:exaeasy},
rational approximations of irrational numbers -- so that both numerator and denominator are odd integers -- are the decisive objects to understand. 
Put differently, 
\begin{equation}\label{eq:defh}
   h(t) \coloneqq \modulus{2+\e^{\i \pi t}+\e^{\i \pi \alpha t}} \quad(t\in \bbR) 
\end{equation}
for irrational $\alpha$  can serve as a means to quantify the error for odd/odd rational approximations of $\alpha$ in the following sense.

\begin{theorem}\label{thm:upplowh}
    If $\alpha\in \bbR$ is irrational, then there exist $c,C>0$ such that
    \begin{equation}
        \label{eq:upplowh}
        c   \left(\min_{u \textnormal{ odd}}\modulus{v \alpha - u}\right)^2
        \leq \inf_{t \in [v-1, v + 1]} h(t)
        \leq C   \left(\min_{u \textnormal{ odd}} \modulus{v \alpha - u}\right)^2
    \end{equation}
    for all odd integers $v$.
\end{theorem}

The proof is done in two steps -- first for the upper and then for the lower bound.

\begin{proof}[Proof of the upper estimate in Theorem~\ref{thm:upplowh}]
    Fix an odd $v \in \Z$ and choose an odd $u\in \Z$ so that $\modulus{v \alpha - u}$ is minimal.
    If 
    \[
    	\modulus{v\alpha - u} 
    	\geq   \min\{\modulus{1 + \alpha},1\}
    	\eqqcolon c_0,
    \] 
    then the upper bound is trivially satisfied for any $C\geq 4 /c_0^2$ because $h(t)\leq 4$.
    On the other hand, if
    \[
    	\modulus{v\alpha - u} 
    	< c_0,
    \] 
    we set
    \[
    	\delta\coloneqq - \frac{v \alpha - u}{1+\alpha}
    	\quad \text{and} \quad
    	t_0 \coloneqq v + \delta.
    \]
    In particular,
    \[
    	\modulus{\delta} 
    	= \frac{\modulus{v\alpha - u}}{\modulus{1+\alpha}}
    	< \frac{c_0}{\modulus{1+\alpha}}
    	\leq 1,
    \]
    which ensures that $t_0 \in [v-1, v + 1]$. 
    Moreover, by construction, we have $\alpha t_0 = u - \delta$.
    Since $u$ and $v$ are odd, we obtain
    \begin{align*}
    	h(t_0)
    	&= \modulus{2+\e^{\i \pi t_0} + \e^{\i \pi t_0\alpha} }
    	= \modulus{2+\e^{\i v\pi + \i \pi \delta} + \e^{\i u\pi - \i \pi \delta}}
    	= \modulus{2- e^{\i\pi\delta} - e^{-\i\pi\delta}}\\
        &= \modulus{2- 2 \cos(\pi \delta)}
        = 4 \sin^2\left( \frac{\pi\modulus{\delta}  }{2}\right)\\
        &\le 4\left( \frac{\pi\modulus{\delta}  }{2}\right)^2 
        = (\pi \delta)^2
    	= \left(\frac{\pi\modulus{v\alpha - u}}{\modulus{1+\alpha}}  \right)^2.
    \end{align*}
    Therefore,
    \[
    	\inf_{t \in [v-1, v + 1]} h(t)	
    	\le h(t_0)
        \le \frac{ \pi^2}{\modulus{1+\alpha}^2} \modulus{v\alpha - u}^2.
    \]
    In either case, using
    \[
    	C \coloneqq \max\left\{ 
    		\frac{4}{c_0^2}, \frac{ \pi^2}{\modulus{1+\alpha}^2}	
    	\right\}
    \]
    yields the desired upper estimate.
\end{proof}

The lower bound estimate requires the following elementary observation.

\begin{lemma}\label{lem:deltas}
    Let $\alpha \in \R$. Then for every $\delta_1, \delta_2 \in \R$
    we have
    \[
    	(\delta_2 - \alpha\delta_1)^2
    	\leq (1+\alpha^2)  (\delta_1^2 + \delta_2^2).
    \]
\end{lemma}

\begin{proof}
    By Cauchy-Schwarz inequality,
    \begin{align*}
    	(\delta_2 - \alpha\delta_1)^2 
        & = \duality{ (1, -\alpha)  }{ (\delta_2,\delta_1)}^2_2
         \le (1+\alpha^2) (\delta_1^2 + \delta_2^2).
        \qedhere
    \end{align*}
\end{proof}

\begin{proof}[Proof of the lower estimate in Theorem~\ref{thm:upplowh}]
    Let $v\in \bbZ$ be odd and fix $t \in [v-1, v + 1]$.
    Then we can write
    \begin{align}
     t &= v + \delta_1,
    	\label{eq:th}\\
    	t\alpha &= u + \delta_2,
    	\quad \text{with } u \in \Z \text{ odd and } \delta_1, \delta_2 \in [-1, 1].
    	\label{eq:talphah}
    \end{align}
    Since $u$ and $v$ are odd integers, we obtain
    \begin{equation}\label{eq:ht}
        h(t) 
    	=  \modulus{  2 +  e^{\i \pi(v + \delta_1)} + e^{\i\pi (u + \delta_2)}  } 
    	= \modulus{ e^{\i\pi \delta_1} + e^{\i\pi \delta_2} - 2 }.
    \end{equation}
    	
    Let $\epsilon \in (0,1)$ be such that $\epsilon e^{ \epsilon} =1$.
    By continuity,
    \[
        c_1\coloneqq\inf \left\{h(t)\colon t \text{ satisfies~\eqref{eq:th} and~\eqref{eq:talphah} with }
    	\max_i    
    		\modulus{\delta_i}
    	 \geq \epsilon \pi^{-1}\right\}
        >0.
    \]
    Since $\min_{u \textnormal{ odd}}\modulus{v \alpha - u} \leq 1$, we obtain that
    \[
        h(t) \ge c_1 \ge c_1   \left(\min_{u \textnormal{ odd}} \modulus{v \alpha - u}\right)^2
    \]
    whenever $\pi\modulus{\delta_1} \geq \epsilon$ or $\pi\modulus{\delta_2} \geq \epsilon$.
    
    Let us now consider the case that $\modulus{\delta_i'}:=\pi\modulus{\delta_i} < \epsilon$ for $i=1,2$. Due to the second order Taylor expansion and the fact that $y e^{y} \le \epsilon e^{\epsilon} = 1$ for $y \le \epsilon$, we have
    \begin{equation}\label{eq:Taylor}
        \modulus{\e^{\i x} - (1 + \i x - x^2/2)}
        \le \frac{\modulus{x}^3}{6} e^{\modulus{x}} 
        \le \frac{x^2}6,
    \end{equation}
        
    whenever $\modulus{x} \le \epsilon$.
    From \eqref{eq:ht} and \eqref{eq:Taylor} we obtain    
    \begin{align*}
    	h(t)
    	& \geq\modulus{ \i (\delta_1' + \delta_2') - \frac{1}{2}(\delta_1'^2 + \delta_2'^2)}-\frac{1}{6}(\delta_1'^2 + \delta_2'^2) \\
        & = \sqrt{(\delta_1' + \delta_2')^2+\frac14 (\delta_1'^2 + \delta_2'^2)^2  }-\frac{1}{6}(\delta_1'^2 + \delta_2'^2) \\
        & \geq \frac{1}{2}(\delta_1'^2 + \delta_2'^2) - \frac{1}{6}(\delta_1'^2 + \delta_2'^2) \\
        & =  \frac{1}{3}(\delta_1'^2 + \delta_2'^2)
         =  \frac{\pi^2}{3}(\delta_1^2 + \delta_2^2)\\
        & \geq
    	 \frac{\pi^2}{3(1+\alpha^2) }  |\delta_2 - \alpha\delta_1|^2,
    \end{align*}
    where the last inequality is true due to Lemma~\ref{lem:deltas}.
    Since $\delta_2- \alpha \delta_1 = v \alpha - u$, it follows that
    $
    	h(t) \ge c_2  \modulus{v\alpha - u}^2
    $
    with $c_2 \coloneqq \frac{\pi^2}{3(1+\alpha^2) }$.

    In either case, setting $c \coloneqq \min\{c_1, c_2\}$ yields the desired lower bound on $h$.
\end{proof}

\section{Proof of Proposition~\ref{prop:square}}
    \label{sec:proof-square}

    We recall $T_{t,\alpha}$ from \eqref{eq:Talphat}. Due to Theorem~\ref{thm:MforPhs}, both assertions in Proposition~\ref{prop:square} require us to estimate $\|T_{t,\alpha}^{-1}\|$. To this end -- since $\|T_{t,\alpha}\|$ is uniformly bounded in $t \in \bbR$ -- by Cramer's rule, it suffices to bound $1/\det(T_{t,\alpha})$. However, in order to not overcomplicate the proofs, instead of $1/\det(T_{t,\alpha})$ we consider $\det(T_{t,\alpha})$ directly in the following. Therefore, in order to obtain the second assertion in Proposition~\ref{prop:square}, the question becomes when does
    \[
    	g(t) \coloneqq \modulus{ 1 + \frac{1}{2} \left( \e^{\i t} + \e^{\i t\alpha} \right) }
    \]
    get small for particular values of $t$ as $t \to \infty$. Note that this expression makes sense for any $\alpha\in \bbR$ and that $g(t) = \frac{1}{2}h(t/\pi)$ with $h$ from \eqref{eq:defh}. The next lemma gives the answer as well as enables us to disprove exponential stability in Proposition~\ref{prop:strongstab}.

\begin{lemma}\label{lem:upperboundt2}
    Let $\alpha \in \R$ be irrational. Then there is $C>0$ such that for any $t_0 \in \R$ we find $t\geq t_0$ with
    $
    	g(t) 
    	\leq \frac{C}{t^2}.
    $
    \end{lemma}

The proof of the lemma is based on the fact that any irrational number can be approximated quadratically well by fractions of the shape odd/odd (Lemma~\ref{lem:odd-odd-approx}). Recall from the previous section that the behaviour of the function $h$ (and, thus, of $g$) depends on the exact approximation properties of $\alpha$ by fractions of the shape odd/odd. A more detailed view of this is described in the next lemma. Lemma~\ref{lem:upperboundt2} follows immediately from Lemma~\ref{lem:h_ub_new} by setting $\psi(t) = 1/t$ and using Lemma~\ref{lem:odd-odd-approx}.

\begin{lemma}\label{lem:h_ub_new}
    Let $\alpha\in \R$ be irrational and $\psi \colon (0,\infty) \to (0, \infty)$ decreasing. Assume that there exist infinitely many pairs of odd integers $u,v$ such that
    \[
    	\modulus{ \alpha - \frac{u}{v} }
    	\leq \frac{\psi(v)}{v}.
    \]
    Then there exists $C>0$ such that for any $t_0 \in \R$ we find $t\geq t_0$ with
    \[
    	g(t) 
    	\leq C    \left( \psi(\tfrac{t}{\pi}-1) \right)^2.
    \]
\end{lemma}
\begin{proof}
    First, note that we may assume that $\psi(t) \to 0$ for $t\to \infty$ because otherwise, the statement is trivial.
    By assumption we have
    \[
        \modulus{v\alpha - u}^2 \leq (\psi(v))^2
    \]
    for infinitely many odd integers $u,v$.
    In particular, this means that we have infinitely many odd integers $1\leq v_1 < v_2 < \ldots$ which satisfy the above inequality for some odd $u_n$'s.
    Recall the function $h$ from~\eqref{eq:defh}.
    Compactness of the interval $[v_n-1,v_n+1]$ yields $t_n \in [v_n-1,v_n+1]$ such that
    \begin{align*}
        h(t_n)
        = \inf_{t \in [v_n-1,v_n+1]} h(t)
        &\leq C   \left(\min_{u \text{ odd}} |v_n \alpha - u|\right)^2\\
        &\leq C   (\psi(v_n))^2
        \leq C   (\psi(t_n-1))^2;
    \end{align*}
    where the first inequality is the upper bound from Theorem~\ref{thm:upplowh}.
    Since $t_n \to \infty$ for $n\to \infty$, reformulating the statement in terms of $g$, the assertion follows.
\end{proof}

Next, we want to prove a lower bound for $g(t)$, which holds for all sufficiently large $t$. It will again depend on the approximation properties of $\alpha$. 

\begin{lemma}\label{lem:lowerboundpsi_new}
    Let $\alpha \in \R$ be irrational and $\psi \colon (0,\infty) \to (0, \infty)$ be a decreasing function such that the inequality
    \[
    	\modulus{ \alpha - \frac{u}{v} }
    	< \frac{\psi(v)}{v}
    \]
    has finitely many solutions in odd integers $u, v$.
    Then there exists $c>0$ such that
    \begin{equation}\label{eq:lowerboundpsig}
    	c    (\psi(\tfrac{t}{\pi}+1))^2 
    	\leq  g(t)\qquad (t\ge 0).
    \end{equation}
\end{lemma}

Because of Lemma~\ref{lem:Khintchin}, the assumptions of Lemma~\ref{lem:lowerboundpsi_new} are satisfied, in particular, for $\psi(t) = 1/(t (\log t)^{1+\eps})$ and for almost all $\alpha$.

\begin{proof}[Proof of Lemma~\ref{lem:lowerboundpsi_new}]
    By assumption, we can find a constant $c_1>0$ such that
    \[
        \modulus{ v \alpha - u }
    	\geq c_1   \psi(v)
    \]    
    for all odd integers $u$ and $v$.
    For any $t\geq 0$, let us denote by $v_t$, the closest odd integer to $t$. Then by Theorem~\ref{thm:upplowh} -- using $h$ as in \eqref{eq:defh} -- we can find $c_2>0$ such that
    \begin{align*}
        h(t)
        \geq \inf_{s\in [v_t-1, v_t+1]} h(s)
        &\geq c_2   \left(\min_{u \text{ odd}} |v_t \alpha - u|\right)^2\\
        &\geq c_2   c_1^2   (\psi(v_t))^2
        \geq c_2   c_1^2   (\psi(t + 1))^2.
    \end{align*}
    Reformulation in terms of $g$ instead of $h$ yields the result.
\end{proof}

Setting $\psi(t) \coloneqq 1/(t (\log t)^{1+\frac{\eps'}{2}})$ for some $0<\eps'<\eps$ and using Lemma~\ref{lem:Khintchin}, we immediately get the following metric result for the lower bound.

\begin{lemma}\label{lem:lowerboundae}
For Lebesgue a.e.~$\alpha \in \R$  and all $\eps>0$, there is $t_0>0$ such that 
\begin{equation*}
	\frac{1}{t^2 (\log t)^{2 + \eps}} 
	\leq g(t) \quad(t\geq t_0).
\end{equation*}
\end{lemma}

\begin{proof}[Proof of Proposition~\ref{prop:square}] 
    Keeping the discussion at the beginning of the section in mind, both assertions follow by appealing to Lemmata~\ref{lem:upperboundt2} and~\ref{lem:lowerboundae}.
\end{proof}

\section{Proof of Proposition~\ref{prop:othertypes}}
    \label{sec:proof-othertypes}

The result in Proposition~\ref{prop:square} only qualifies for Lebesgue-almost every $\alpha\in \R$. In fact, one can construct specific $\alpha$'s such that $g(t)$ decays arbitrarily quickly for the ``worst'' $t$'s. This is the content of the present section. 
Recall that
\[
    g(t) =  \modulus{ {1 + \frac{1}{2}(\e^{\i t} + \e^{\i t\alpha})}}\qquad(t\in \R).
\]

\begin{lemma}\label{lem:construction}
Let $f \colon (0, \infty) \to (0, \infty)$ be decreasing. There exists an irrational $\alpha >0$ such that the following assertions hold for some $c,C>0$.
\begin{enumerate}[\upshape (a)]
\item For all $t_0>\pi$, there is $t\geq t_0$ with
\begin{equation}\label{eq:const_ub}
	g(t)
	\leq C  f(t-\pi).
\end{equation}
\item If $\lim_{t\to\infty} f(t)  t^2 = 0$, thenthere is $t_1>0$ such that for all $t\geq t_1$ we have
\begin{equation}\label{eq:const_lb}
	 c  f(t+\pi)
	\leq g(t).
\end{equation}
\end{enumerate}
\end{lemma}

The proof of Lemma~\ref{lem:construction} is based on  continued fraction expansions for real numbers. In a particular sense, the continued fraction expansion gives the optimal approximations of irrational numbers by rationals. Some results of the theory of Diophantine approximation are gathered in Appendix~\ref{appendix:diophantine}.

\begin{definition}\label{def:continued_frac}
The \emph{continued fraction} of an $\alpha \in \bbR$ is an expression of the form
\[
	a_0+\cfrac{1}{a_1 +\cfrac{1}{a_2 +\cfrac{1}{\ddots + \cfrac{1}{a_n}}}}
	\qquad \mbox{or} \qquad
	a_0+\cfrac{1}{a_1+\cfrac{1}{a_2+\cfrac{1}{a_3+\ddots}}}
\]
where $a_0,a_1,a_2,\dots$ are obtained from the \emph{continued fraction algorithm}:
Set $a_0\coloneqq\floor{\alpha}$. If $a_0\neq \alpha$, write $\alpha = a_0 + {1}/{\alpha_1}$, where $\alpha_1>1$ and set $a_1\coloneqq\floor{\alpha_1}$. Next, if $a_1\neq \alpha_1$, we write $\alpha_1=a_1+{1}/{\alpha_2}$, where $\alpha_2>1$ and set $a_2\coloneqq\floor{\alpha_2}$. Continue this process until $a_n=\alpha_n$ for some $n$. If $\alpha$ is rational, the process terminates and we obtain a finite continued fraction, which is written as $[a_0;a_1,a_2,\dots,a_n]$. If $\alpha$ is irrational, then the process does not terminate and we obtain an infinite continued fraction, written as $[a_0;a_1,a_2,\dots]$. 
The rationals
$
	\frac{p_j}{q_j}=[a_0;a_1,a_2,\dots,a_j],
$
for relatively prime integers $p_j, q_j$ with $q_j>0$, are called the \emph{convergents to $\alpha$}.
For irrational $\alpha$ the convergents indeed converge to $\alpha$, and there is a one-to-one correspondence between the irrationals and the infinite continued fraction expansions.
\end{definition}

\begin{proof}[Proof of Lemma~\ref{lem:construction}]
(a) We set $\alpha = [1; a_1, a_2, \ldots]$ by recursively defining
\[
	a_{n} \coloneqq 2\ceil[\Big]{\frac{1}{\sqrt{f(\pi q_{n-1})}  q_{n-1}}}    
	\quad \text{for } n \geq 1,
\]
where $q_n = a_nq_{n-1}+q_{n-2}$ for $n\geq 2$ and $q_0=1$ and $q_1 = a_1$, as in Lemma~\ref{lem:cont_frac_recurrence}.
We turn to (a) first and show~\eqref{eq:const_ub} for infinitely many, arbitrarily large $t$. 
By Lemma~\ref{lem:cont_frac_bounds}, we have for all convergents that
\begin{equation}
    \label{eq:construction-ub}
	\modulus{ \alpha - \frac{p_n}{q_n} }
	< \frac{1}{q_n^2 a_{n+1}}
	< \frac{\sqrt{f(\pi q_{n})}}{q_n}.
\end{equation}
Moreover, since $a_n$ is even for $n\geq 1$, one can check that by Lemma~\ref{lem:cont_frac_recurrence} the convergents are of the following shapes: if $n$ is even, then $p_n/q_n = \odd/\odd$; if $n$ is odd, then $p_n/q_n = $ odd/even.
Taking only the convergents of the shape odd/odd, we get infinitely many fractions $u/v = \odd/\odd$ with
\[
	\modulus{ \alpha - \frac{u}{v} }
	< \frac{\sqrt{f(\pi v)}}{v}.
\]
Then Lemma~\ref{lem:h_ub_new} immediately implies
that there exists $C>0$ such that for any $t_0 >\pi$ there exists $t\geq t_0$ with
$
    g(t)
	\leq C  f(t-\pi).
$

(b)  We want to employ Lemma~\ref{lem:lowerboundpsi_new}. For this purpose, we show that
the inequality
\begin{equation}\label{eq:ex_cprime}
	\modulus{ \alpha - \frac{u}{v} }
	< \frac{\sqrt{f(\pi v)}/4}{v}
\end{equation}
has only finitely many solutions in odd integers $u, v$.

First, note that since $f(t)  t^2 \to 0$ for $t\to \infty$,
the right-hand side of~\eqref{eq:ex_cprime} becomes less than $1/(2v^2)$ for sufficiently large $v$. Thus, Lemma~\ref{lem:Legendre} implies that every solution $u/v$ to~\eqref{eq:ex_cprime}, with $v$ sufficiently large, has to be a convergent to $\alpha$.
Therefore, it suffices to check that only finitely many convergents $u/v=p_n/q_n$ satisfy~\eqref{eq:ex_cprime}.
The lower bound for approximations by convergents from Lemma~\ref{lem:cont_frac_bounds} tells us that
\begin{equation}\label{eq:anplus1}
\frac{1}{(a_{n+1}+2)   q_n^2}
	< \modulus{ \alpha - \frac{p_n}{q_n} }.
\end{equation}
Note that
\[
	(a_{n+1}+2)  q_n^2 
	= \left(2\ceil[\Big]{\frac{1}{q_{n}\sqrt{f(\pi q_{n})} }}+2 \right)  q_n^2
	\leq 3  \frac{q_n}{\sqrt{f(\pi q_{n})}}
\]
for sufficiently large $q_n$. Thus,~\eqref{eq:anplus1} in conjunction with~\eqref{eq:construction-ub} gives
\[	
\frac{\sqrt{f(\pi q_{n})}}{3q_n}
	< \modulus{ \alpha - \frac{p_n}{q_n} }
	< \frac{\sqrt{f(\pi q_n)}}{4q_n},
\]
which is clearly impossible.
Therefore, inequality~\eqref{eq:ex_cprime} only has finitely many solutions $u/v$, and Lemma~\ref{lem:lowerboundpsi_new} implies that there is $c_1>0$ and $t_1>0$ such that
\begin{equation*}
	\frac{c_1}{16}  f(t+\pi)
	\leq g(t) \qquad(t\geq t_1).
    \qedhere
\end{equation*}
\end{proof}

\begin{example}\label{ex:alpha_wellapprox}
For $f(t) = e^{-t}$ in Lemma~\ref{lem:construction}, $\alpha$ is constructed  by setting $\alpha = [1; a_1, a_2, \ldots]$; where
\[
	a_{n} \coloneqq 2\ceil[\Big]{\frac{1}{e^{-\pi q_{n-1}/2}  q_{n-1}}}   
	\quad \text{for } n \geq 1.
\]
Then we have
\[
	c  e^{-t}
	\leq 
	g(t)
	\leq
	C   e^{-t};
\]
where the left inequality holds for all $t$ sufficiently large, and the right inequality is satisfied for infinitely many $t$ accumulating at $\infty$.
\end{example}

\begin{proof}[Proof of Proposition~\ref{prop:othertypes}]
    Recalling the discussion at the beginning of Section~\ref{sec:proof-square} and
    setting $f=\frac{1}{\gamma}$, the statement is a direct consequence of Lemma~\ref{lem:construction}.
\end{proof}

\section{Proof of Proposition~\ref{prop:badsquare}}
    \label{sec:posinc}

For convenience, let us consider
\[
    \widetilde{m}_\alpha(\eta)
    \coloneqq m_\alpha(\pi\eta) /2
    = \sup_{t \in [-\eta, \eta]} h(t)^{-1};
\]
where
$
    h(t) \coloneqq \modulus{2+\e^{\i \pi t}+\e^{\i \pi \alpha t}}.
$
Clearly, $\widetilde{m}_\alpha(\eta)$  grows quadratically if and only if $m_\alpha(\eta)$ does.
It is a basic fact (due to Lemma~\ref{lem:cont_frac_bounds}) that $[a_0; a_1, a_2, \ldots]$ is badly approximable (see Definition before Theorem~\ref{thm:decayest}) if and only if $(a_n)_n$ is bounded. 

The decisive information establishing Proposition~\ref{prop:badsquare}(a) is the following.

\begin{lemma}\label{lem:oddoddconsant} 
    Let $\alpha\in \bbR$ be badly approximable and let $(v_n)$ be a strictly increasing sequence such that
    \[
        \{v_n:n\in \N\} = \{ v\in \N: v\textnormal{ odd and there is odd  $u\in \N$ with } \modulus{v\alpha-u}\leq 2/v\}.
    \]
    Then there exists $C\geq 0$ such that 
    $
       v_{n+1}\leq C v_n
    $    
    for all $n\in \N$.
\end{lemma}

\begin{proof}
     As $\alpha = [a_0; a_1, a_2, \ldots]$ is badly approximable, $(a_n)$ is bounded by, say $C_1>0$.
    Fix an integer $n\geq 2$ and denote by $(q_k)_k$ the sequence of denominators of the convergents of $\alpha$. Since $q_k\to\infty$, we find $k\in \N$ such that $q_{k-1}\leq v_n < q_{k}$. By Lemma~\ref{lem:oddapp1}, there exist $u,v\in \N$ odd such that $q_{k}\leq v\leq q_{k+2}$ and $\modulus{v\alpha-u}\leq 2/v$.
    Thus, $v> v_n$. As $(v_n)$ is strictly increasing, we get $v_{n+1} \le v\leq q_{k+2}$. 
    
    Since $(q_k)_k$ is strictly increasing,  the recurrence relation in Lemma~\ref{lem:cont_frac_recurrence} gives
    \[
        q_{k+2} = a_{k+2}q_{k+1}+q_k < (C_1+1) q_{k+1}.
    \]
    A recursive application of the above inequality implies
    \[
        q_{k+2} \leq (C_1+1)^3 q_{k-1}.
    \]

    Combining the observations in the previous two paragraphs gives $v_{n+1}\leq q_{k+2}\leq (C_1+1)^3 q_{k-1}\leq  (C_1+1)^3 v_n$, establishing the assertion.
\end{proof}

For $\alpha\in \bbR$, we say that $\alpha=[a_0;a_1,a_2,\ldots]$ has \emph{large odd/odd gaps}, if for any $C>0$ there exists $n\in \N$ such that the $n$-th convergent is of the shape $p_n/q_n = \odd/\odd$ and $a_{n+1}\geq C$. This property is generic in the Lebesgue measure sense (Lemma~\ref{lem:notposdense}). 

\begin{example}
    \label{ex:odd-odd-gaps}
  The number $\alpha$ from Example~\ref{ex:alpha_wellapprox} has large odd/odd gaps because $a_n \to \infty$ for $n \to \infty$, and every other convergent is odd/odd.  
\end{example}

\begin{lemma}\label{lem:not_pos_incr}
If $\alpha \in \bbR$ has large odd/odd gaps, then 
 $m_\alpha$ is not of positive increase.
\end{lemma}
\begin{proof}
Again, for simplicity, we consider
\[
    \widetilde{m}_\alpha(\eta)
    \coloneqq m_\alpha(\pi\eta) /2
    = \sup_{t \in [-\eta, \eta]} h(t)^{-1},
\]
as clearly $m_\alpha(\eta)$ is of positive increase if and only if $\widetilde{m}_\alpha(\eta)$ is.

Assume towards a contradiction that $\widetilde{m}_\alpha(\eta)$ is of positive increase, i.e., 
there exist $\beta>0, c\in (0,1]$, and $\eta_0>0$ such that for all $\lambda \geq 1$ and $\eta\geq \eta_0$, we have
\begin{equation}\label{eq:posincm}
    \widetilde{m}_\alpha(\lambda \eta) 
    \geq c \lambda^{\beta} {\widetilde{m}_\alpha(\eta)}.
\end{equation}
By Theorem~\ref{thm:upplowh}, we find $c_1, C_1 > 0$ such that for all odd integers $v$ we have
\[
    c_1   \left(\min_{u \textnormal{ odd}}\modulus{v \alpha - u}\right)^2
    \leq \inf_{t \in [v-1, v + 1]} h(t)
    \leq C_1   \left(\min_{u \textnormal{ odd}} \modulus{v \alpha - u}\right)^2.
\]
We choose $\lambda>1$ such that
\begin{equation}\label{eq:set_lambda}
   c \lambda^\beta > C_1/c_1.   
\end{equation}
Since $\alpha$ has large odd/odd gaps, there is $n\in \N$ such that the $n$-th convergent is of the shape $p_n/q_n = \odd/\odd$ and 
$
	a_{n+1} > 2 \lambda.
$
Without loss of generality, assume $q_n\geq \eta_0$, otherwise choose a larger $n$, and set $\eta \coloneqq q_n + 1$. 
Then by Theorem~\ref{thm:upplowh},
\[
	\inf_{t\in [-\eta,\eta]} h(t)
	\leq \inf_{t \in [q_n-1, q_n+1]} h(t)
	\leq C_1   \modulus{q_n \alpha - p_n}^2,
\]
and thus
\begin{align}\label{eq:malphaeta_lb}
	\widetilde{m}_\alpha(\eta)
	= \sup_{t\in [-\eta,\eta]} h(t)^{-1}
	\geq C_1^{-1} \modulus{q_n \alpha - p_n}^{-2}.
\end{align}
On the other hand, note that by Lemma~\ref{lem:best_approximations} for all integers $u,v$ with $0<|v|<q_{n+1}$ we have $\modulus{\alpha v - u} \geq \modulus{q_n \alpha - p_n}$.
By Lemma~\ref{lem:cont_frac_recurrence}, we have $q_{n+1} = a_{n+1} q_n + q_{n-1} > 2 \lambda q_n +1 \geq \lambda (q_n + 1)+1 = \lambda \eta +1$ and in turn,
\[
	\modulus{\alpha v - u} \geq \modulus{q_n \alpha - p_n}
	\quad \text{for all } |v| \leq \lambda \eta +1.
\]
Thus, Theorem~\ref{thm:upplowh} implies
$
	\inf_{t\in [-\lambda \eta,\lambda \eta]} h(t)
	\geq c_1   \modulus{q_n \alpha - p_n}^2,
$
and so 
\begin{equation}\label{eq:malphaeta_ub}
	\widetilde{m}_\alpha(\lambda \eta)
	\leq c_1^{-1}   \modulus{q_n \alpha - p_n}^{-2}.
\end{equation}
Combining \eqref{eq:malphaeta_ub}, \eqref{eq:malphaeta_lb}, and \eqref{eq:set_lambda}, we obtain
\[
	\widetilde{m}_\alpha(\lambda \eta)
	\leq c_1^{-1}   \modulus{q_n \alpha - p_n}^{-2}
    \leq c_1^{-1} C_1 \widetilde{m}_\alpha(\eta)
    < c \lambda^{\beta} \widetilde{m}_\alpha(\eta), 
\]
contradicting \eqref{eq:posincm}.
\end{proof}

\begin{lemma}\label{lem:notposdense}
    For Lebesgue a.e.~$\alpha \in \bbR$, $\alpha$ has large odd/odd gaps.
\end{lemma}
\begin{proof}
    By \cite[Theorem 9.2]{Bugeaud2012}, for Lebesgue a.e.~$\alpha \in \bbR$, $\alpha$ has a \emph{normal continued fraction} -- in particular, every block of positive integers $d_1, \ldots, d_k$ occurs in the continued fraction expansion $[a_0; a_1, a_2, \ldots]$. Let $\alpha\in \R$ have a normal continued fraction. Then for any odd $C\in \N$, there exists $n\in \N$ such that $a_{n+1} = a_{n+2} = a_{n+3}=C$. So, at least one of the convergents $p_{n}/q_{n}, p_{n+1}/q_{n+1}, p_{n+2}/q_{n+2}$ must be of the shape odd/odd. Indeed, assume by contradiction that none of them are of the shape odd/odd. Then, by Lemma~\ref{lem:oddapp0}, both consecutive denominators as well as consecutive numerators need to change parity. Thus, only the two cases, $p_n/q_n$ odd/even, $p_{n+1}/q_{n+1}$ even/odd, and $p_{n+2}/q_{n+2}$ odd/even along with $p_n/q_n$ even/odd, $p_{n+1}/q_{n+1}$ odd/even, and $p_{n+2}/q_{n+2}$ even/odd need to be rendered impossible. Recalling Lemma~\ref{lem:cont_frac_recurrence}, we obtain the two equalities
    \[
         q_{n+2} = a_{n+2}q_{n+1} +q_{n} \quad \text{ and } \quad
        p_{n+2} = a_{n+2}p_{n+1} +p_{n}.
    \]
    The first equality, as $a_{n+2}$ is odd, yields if $q_n$ and $q_{n+2}$ both are even, then so is $q_{n+1}$, excluding the first case. Whereas the second equality necessitates $p_{n+1}$ to be even, if both $p_{n+2}$ and $p_{n}$ are, eliminating the second case as well. Hence, we find $\ell\in \{n,n+1,n+2\}$ such that $p_\ell/q_\ell$ is of the shape odd/odd. Since $a_{\ell+1}=C$ can be chosen arbitrarily large, we eventually establish the assertion.
\end{proof}

\begin{proof}[Proof of Proposition~\ref{prop:badsquare}]
    Firstly,~(b) is a direct consequence of Lemmata~\ref{lem:not_pos_incr} and~\ref{lem:notposdense}.
    
    (a) Let $\alpha\in (0,\infty)$ be badly approximable. We want to show that there exists $c,C,\eta_0>0$ such that
    \[
       c\eta^2 \leq  \widetilde{m}_\alpha (\eta) \leq C\eta^2\qquad \text{for all }\eta \ge \eta_0.
    \]
    As $\alpha$ is badly approximable, the assumptions of Lemma~\ref{lem:lowerboundpsi_new}
    are fulfilled with $\psi(v)=\tilde cv^{-1}$ for some $\tilde c>0$. The required upper bound can therefore be obtained by Lemma~\ref{lem:lowerboundpsi_new}, again owing to the discussion at the beginning of Section~\ref{sec:proof-square}.
    
    We now prove the lower bound. Since $\alpha$ is irrational, by Lemma~\ref{lem:odd-odd-approx_ingrid}, we find a strictly increasing sequence $(v_n)_n$ of odd integers such that
    for each $n\in \bbN$, there exists an odd $u_n\in \bbN$ with
    \[
        \modulus{v_n\alpha-u_n}\leq \frac{2}{v_n}.
    \]
    As $\alpha = [a_0; a_1, a_2, \ldots]$ is badly approximable, by Lemma~\ref{lem:oddoddconsant}, we find $C\geq 0$ such that
    $
    	v_{n+1} 
    	\leq Cv_n
    $
    for all $n\in \N$.
    
    Next, let $\eta \geq 2$ and $\ell\in \N$ be maximal with $v_{\ell} \leq  \eta - 1$.
    It follows from our choice of $(v_n)$ and $(u_n)$ and Theorem~\ref{thm:upplowh}
    that there is $C_2\geq 0$ such that 
    \begin{align}
        \inf_{t \in [-\eta, \eta]} h(t)
        &\leq C_2   \left(\min_{
            \substack{
            v \in [-(\eta -1), \eta -1] \textnormal{ odd}\\
            u \textnormal{ odd}
            }
            } \modulus{v \alpha - u}\right)^2 \nonumber\\
        & \leq C_2   |v_{\ell} \alpha - u_{\ell}|^2
        \leq  \frac{2C_2}{v_{\ell}^2}.
            \label{eq:infg_1}
    \end{align}
    Further, as $(v_n)_n$ is strictly increasing, our choices of $\ell,C$ ensure that $ \eta - 1 < v_{\ell + 1} \leq C v_{\ell}$. This implies 
    $
        v_{\ell} \geq c_1 \eta,
    $
    for $c_1 \coloneqq (2 C)^{-1}$. 
    Combining this with \eqref{eq:infg_1}, we infer
    \[
        \inf_{t \in [-\eta, \eta]} h(t)
        \leq   \frac{2C_2 }{(c_1 \eta)^2}
        =   \frac{C_3 }{\eta^2};
    \]
    where $C_3 = 2 C_2 c_1^{-2}$.
    This settles the lower bound for $\widetilde{m}_\alpha (\eta)$.
\end{proof}

\section{Conclusion}

We have revisited the stability problem for port-Hamiltonian systems and pro\-vided an estimate of the resolvent growth function in terms of the matrix norm of a suitable inverse of a derived quantity with which it was possible to characterise exponential stability and asymptotic stability in earlier findings. The precise estimate led to the consideration and analysis of a rather elementary port-Hamiltonian system, where many different examples for possible decay of the port-Hamiltonian semigroup can be provided. The core observation was that the resolvent growth can be reformulated in terms of Diophantine approximation results, that is, how well irrational numbers can be approximated by rationals.

\section*{Acknowledgements}

We thank David Seifert for useful a discussion relating to the optimality of the decay estimates for $C_0$-semigroups. The first author was funded by the Deutsche Forschungsgemeinschaft (DFG, German Research Foundation) -- 523942381. The third author was funded by the Austrian Federal Ministry of Education, Science and Research (BMBWF), grant no. SPA 01-080 MAJA.

\bibliographystyle{abbrvurl}
\bibliography{literature}

\pagebreak
\appendix

\section{Background on quantified semi-uniform stability}
    \label{appendix:stability-background}

This appendix is dedicated to recalling known results about quantitative decay rates for $C_0$-semigroups in terms of the resolvent. These results are applied to study non-uniform stability of a particular class of semigroups in Section~\ref{sec:exaeasy}. While versions of Theorem \ref{thm:lowerbound} below are known, and the statement as formulated here maybe be folklore for experts, we provide an explicit proof.

\begin{remark}
    \label{rem:inverse-log-sandwich}
    Let $\gamma, \delta:\bbR_+ \to (0,\infty)$ be increasing continuous functions such that there exist $c, C, t_0>0$ such that
    \[
        c \gamma(t) \le \delta(t) \le C \gamma(t) \text{ for all }t\ge t_0.
    \]
    \begin{enumerate}[\upshape (a)]
    \item For sufficiently large $s$,
    \[
        \gamma^{-1}\left(\frac{s}{C}\right)\le \delta^{-1}(s) \le \gamma^{-1}\left(\frac{s}{c}\right).
    \]

    \item There exists $d, D>0$ such that
    \[
        d \gamma_{\log}(\eta) \le \delta_{\log}(\eta) \le D \gamma_{\log}(\eta)
    \]
    for sufficiently large $\eta$.
    \end{enumerate}
\end{remark}

Throughout, let $(U(t))_{t\ge 0}$ be a strongly continuous semigroup of bounded linear operators on a Hilbert space $H$, with infinitesimal generator $-A$. A semigroup $(U(t))_{t\ge 0}$ is called bounded if $\sup_{t\ge0}\norm{U(t)}<\infty$. Denoting the resolvent set of $A$ by $\rho(A)$, we write  $\Res(\lambda,-A)\coloneqq (\lambda+A)^{-1}$ for the resolvent of $-A$ at $\lambda \in \rho(A)$. Assuming $i\bbR \subseteq \rho(A)$, introduce for $\eta\geq 0$, the function
\begin{equation}\tag{M}\label{eq:M}
    M(\eta)\coloneqq \sup_{t\in [-\eta,\eta]} \norm{\Res(\i t,-A)} 
\end{equation}
associated to the resolvent and
\[
  M_{\log}(\eta) \coloneqq M(\eta)\big(\log(1+M(\eta))+\log(1+\eta)\big).
\]
We recall the following celebrated quantified estimate for semi-uniform stability.

\begin{theorem}[{{Batty--Duyckaerts, \cite{BattyDuyckaerts2008}; see the formulation in \cite[Theorem~1.2]{BorichevTomilov2010}}}]
\label{thm:borichev-tomilov}
If $(U(t))_{t\ge 0}$ is bounded and $\i\bbR \subseteq \resSet(A)$, then there exist $c,t_0>0$ such that
\[
    \norm{U(t)A^{-1}}\leq \frac{c}{M_{\log}^{-1}(t/c)}\quad \text{for all }t\geq t_0.
\]
\end{theorem}

The other main source on characterising stability rates we present here is the main result of \cite{RozendaalSeifertStahn2019}. For this, recall that a continuous function $\gamma:\bbR_+\to (0,\infty)$ is said to be of \emph{positive increase} if there are $\alpha, t_0>0$ and $c\in (0,1]$ such that
\[
    \frac{ \gamma(\lambda t)}{\gamma(t)} \ge c \lambda^{\alpha} \quad \text{whenever }\lambda\ge 1 \text{ and }t\ge t_0.
\]

\begin{theorem}[{{Rozendaal--Seifert--Stahn, \cite[Theorem~1.1]{RozendaalSeifertStahn2019}}}]\label{thm:RSS11} If $(U(t))_{t\ge 0}$ is bounded, $\i\bbR\subseteq \resSet(A)$, and the function associated to the resolvent in~\eqref{eq:M} is of positive increase, then there exist $c,C, t_0>0$ such that 
\[
    \frac{c}{M^{-1}(t)}\le \norm{U(t)A^{-1}}\leq \frac{C}{M^{-1}(t)}\quad\text{for all }t\geq t_0.
\]
\end{theorem}

In applications it might be challenging to directly identify $M$ to be of positive increase, let alone the precise estimates for $M$. Thus, concerning applications, the following theorem is of somewhat more direct relevance.

\begin{theorem}[{{Rozendaal--Seifert--Stahn, \cite[Theorem~3.2]{RozendaalSeifertStahn2019}}}]\label{thm:RSS32}  Assume that $(U(t))_{t\ge 0}$ is bounded, $\i \bbR\subseteq \rho(A)$, and let $\gamma\colon \bbR_+\to (0,\infty)$ be a continuous and increasing function of positive increase. If 
\[
    \norm{\Res(\i s,-A)}\leq \gamma(\modulus{s}) \quad \text{for all } s\in \bbR,
\] 
then there exist $C, t_0>0$ such that 
\[
    \norm{U(t)A^{-1}}\leq \frac{C}{\gamma^{-1}(t)}\quad\text{for all }t\geq t_0.
\]
\end{theorem}

For a decreasing right-continuous function $N:\bbR_+ \to (0,\infty)$, we define the function $N^*: (0,\infty) \to \bbR_+$ by
\[
    N^*(s) \coloneqq \inf\{t\geq 0: N(t)\leq s\}\quad\text{for }s>0.
\]
This function has the following properties that we use without quoting.
\begin{enumerate}[(i)]
    \item $N(N^*(s))\leq s$:~indeed, fixing $s>0$ and $t_0:=N^*(s)$, we note that, since $N$ is decreasing, every $\tau>t_0$ lies in the set $\{t\ge 0:N(t)\le s\}$. Passing to the limit $\tau\downarrow t_0$, right-continuity of $N$ implies that $N(t_0)\le s$.

    \item $N(N^*(s))=s \Leftrightarrow s\in \ran(N)$:~indeed if $s \in \ran(N)$, then there exists $t\ge 0$ such that $N(t)=s$. Thus, $N^*(s)\le t$ which because $N$ is decreasing implies $N(N^*(s))\ge N(t)=s$.

    \item $N(t)>s \Leftrightarrow t<N^*(s)$.
\end{enumerate}

Apart from Theorem~\ref{thm:RSS11}, we also comment on the optimality of the stability estimates. The following theorem for continuous $N$ is given in \cite[Theorem~4.4.14(b)]{ABHN11}. The proof can be adapted for the right-continuous situation. We include the details for the sake of completion:

\begin{theorem}\label{thm:ABHN44b} 
Assume that $(U(t))_{t\ge 0}$ is bounded and $\i \bbR\subseteq \resSet(A)$. Consider a decreasing right-continuous function $N\colon \bbR_+\to (0,\infty)$ such that $\lim_{t\to\infty} N(t)=0$. 

If $\norm{U(t)A^{-1}}\leq N(t)$ for all $t\geq 0$, then for each $c\in (0,1)$, there exist $s_0,C\geq 0$ such that
\[
    \norm{\Res(\i s,-A)}\leq C N^{*}\left(\frac{c}{\modulus{s}}\right)\quad\text{whenever } \modulus{s}\geq s_0.
\]    
\end{theorem}

\begin{proof}
    Let $s\in \bbR$ and  $x\in \dom A$. For each $t\ge 0$,
    \begin{align*}
        \norm{U(t)x } = \norm{U(t)A^{-1}Ax} 
                       \le N(t) \norm{Ax} 
                       \le N(t)\big( \norm{(A+\i s)x } + \modulus{s}x\big).
    \end{align*}      
    Setting $K\coloneqq \sup_{t\ge 0} \norm{U(t)}$, we thus obtain for each $t\ge 0$ that
    \begin{align*}
        K t \norm{ (A+\i s)x } & \ge \norm{\int_0^t e^{-\i s \tau} U(\tau) (A+\i s)x~d\tau    }
                              = \norm{ x-e^{-\i st} U(t)x   }\\
                             & \ge \norm{x} - \norm{U(t)x}
                              \ge \norm{x} - N(t) \norm{(A+\i s)x } - \modulus{s}N(t)\norm{x}.
    \end{align*}
    Consequently, 
    \[
        \norm{(A+\i s)x } \ge \frac{ 1- \modulus{s} N(t) }{Kt + N(t)} \norm{x}\qquad \text{for all } t\ge 0.
    \]
    
    Fix $c'\in (c,1)$. Since $N$ is right-continuous and decreasing, for $t= N^*\left( c'\modulus{s}^{-1}\right)$, we have $N(t)\le c'\modulus{s}^{-1}$. Plugging this above, we get 
    \begin{align*}
        \norm{(A+\i s)x } \ge  \frac{ 1- c' }{K N^*\left( c'\modulus{s}^{-1}\right) + c'\modulus{s}^{-1}} \norm{x}.
    \end{align*}
    As $x\in \dom A$ was arbitrary and $\i s \in \resSet(A)$, we conclude
    \begin{align*}
        \norm{\Res(\i s, -A)} &\le \frac{K}{1-c'} N^*\left( c'\modulus{s}^{-1}\right) + \frac{c'}{1-c'} \modulus{s}^{-1} \\
                             &\le \frac{K}{1-c'} N^*\left( c\modulus{s}^{-1}\right) + \frac{c'}{1-c'} \modulus{s}^{-1};
    \end{align*}
    the second inequality uses $N^*$ is decreasing and $c'>c$. Now as $\lim_{t\to \infty}N(t)=0$, so $\lim_{\modulus{s}\to \infty} N^*\left( c\modulus{s}^{-1}\right)=\infty$. On the other hand, $ \lim_{\modulus{s}\to \infty} \frac{c'}{1-c'} \modulus{s}^{-1} =0$. Thus we can find $s_0, C\ge 0$ such that $\norm{\Res(\i s, -A)} \le C N^*\left( c\modulus{s}^{-1}\right)$ whenever $\modulus{s} \ge s_0$.
\end{proof}

The above theorem implies the following lower bound estimate for the trajectories (note that this estimate is the one for the lower bound in Theorem~\ref{thm:RSS11} above).

\begin{theorem}\label{thm:lowerbound} 
    Assume that $(U(t))_{t\ge 0}$ is bounded, $\i \bbR\subseteq \resSet(A)$, and that the function associated to the resolvent defined in~\eqref{eq:M} is unbounded.
    
    Then there exist $c,C,t_0>0$ such that
    \[
        \|U(t)A^{-1}\|\geq \frac{c}{M^{-1}(Ct)} \quad\text{for all }t\ge t_0.
    \]
\end{theorem}

The assertion seems to be standard and was already observed in \cite[pg~282]{ABHN11}. However, it was difficult for us to find an explicit proof other than the related \cite[Corollary~6.11]{BattyChillTomilov2016} or \cite[Proposition~5.3]{CPSST23}. For convenience, we provide the proof:

\begin{proof}[Proof of Theorem~\ref{thm:lowerbound}] 
    We define
    \begin{alignat*}{2}
        N(t) &\coloneqq \sup\{\|U(\tau)A^{-1}\|: \tau> t\} \quad &&\text{for } t\ge 0\quad \text{and}\\
        N^*(s) &\coloneqq \inf\{t\geq 0: N(t)\leq s\}\quad&&\text{for }s>0.
    \end{alignat*}
    Since the semigroup $U$ is bounded, the semigroup law yields
    \[
        N(t) = \sup_{r> 0} \norm{U(r+t)A^{-1}}
            \le \norm{U(t)A^{-1}}\sup_{r> 0}\norm{U(r)}
    \]
    for all $t\ge0$. Therefore, in order to prove the lower bound on $\norm{U(t)A^{-1}}$, it suffices to show it for $N(t)$.
    
    Clearly, $N$ is monotonically decreasing and $\lim_{t\to\infty} N(t)=0$. Moreover, $N$ is right-continuous since for $t_n \downarrow t$, the sets $\{\tau > t_n\}$ are increasing in $n$, whence
    \[
        \lim_n N(t_n) 
                      = \sup_{n\in\N}\sup\{ \norm{U(\tau)A^{-1} }: \tau \in  (t_n,\infty)\}
                      = \sup\{\norm{U(\tau)A^{-1}}: \tau> t\} 
                      = N(t).
    \]
    Thus, by Theorem~\ref{thm:ABHN44b}, for each $c\in (0,1)$ there exist $s_0,C_1\geq 0$ such that
    \[
        \norm{\Res(\i s,-A)}\leq C_1 N^{*}\left(\frac{c}{\modulus{s}}\right)\quad\text{whenever } \modulus{s}\geq s_0.
    \]
    
    By increasing $s_0$, we may assume that $N(0)>\frac{c}{s_0}$ or equivalently, $N^*(\frac{c}{s_0})> 0$.
    Letting $s\in \R$ with $\modulus{s}\ge s_0$ and defining
    \[
        C_0:= \sup_{\modulus{t} \le s_0} \norm{\Res(it, -A) }<\infty,
    \]
    for each $\modulus{s}\ge s_0$, the function defined in~\eqref{eq:M} satisfies
    \begin{align*}
       M(\modulus{s})= \sup_{t \in [-\modulus{s},\modulus{s}]}\|\mathcal{R}(it,-A)\| 
                    \le{}& \max\left\{C_0, \sup_{\modulus{t}\ge s_0} C_1 N^{*}\left(\frac{c}{\modulus{t}}\right)\right\}\\ 
                    ={}& \max\left\{C_0, C_1 N^{*}\left(\frac{c}{\modulus{s}}\right)\right\}\\
                    \leq{}&K N^{*}\left(\frac{c}{\modulus{s}}\right)
    \end{align*}
    with $K\coloneqq \max\left\{C_{0}N^{*}(\tfrac{c}{s_{0}})^{-1},C_{1}\right\}$, 
    where we used that $N^*$ is decreasing.
    Choose an arbitrary $C>K$. Since $M$ is increasing, unbounded and continuous, there exists $t_{0}>0$ such that for every $t\ge t_{0}$ there exists $s_{t}\ge s_{0}$ with $Ct=M(s_{t})$. 
    Then for each $t\ge t_0$, we have
    \[
        N^*\left(\frac{c}{s_t}\right)> \frac{K}{C} N^*\left(\frac{c}{s_t}\right)
                                     \ge \frac{1}{C}M(s_t)  
                                     =t
    \]
    or equivalently,
    \[
        N(t)>\frac{c}{s_t} \ge \frac{c}{M^{-1}(Ct)};
    \]
    using the fact that $M^{-1}(Ct) = M^{-1}(M(s_t)) \ge s_t$ from \cite[Lemma~1(g)(6)]{Wacker2023}.
\end{proof}
Finally, we recall an adapted optimality statement from \cite{CPSST23}.

\begin{proposition}[{{\cite[Proposition~5.3]{CPSST23}}}] 
    \label{prop:CPSST23-53}
Suppose that $(U(t))_{t\ge 0}$ is bounded with $\i \bbR\subseteq \rho(A)$. Let $\gamma\colon \bbR_+\to (0,\infty)$ be continuous, increasing, and unbounded. If 
\[
   \limsup_{\modulus{s}\to\infty} \frac{\norm{\Res(\i s,-A)}}{\gamma(\modulus{s})}>0,
\]then there exists $c>0$ such that
\[
   \limsup_{t\to\infty} \gamma^{-1}(ct)\norm{U(t)A^{-1}}>0.
\]
Moreover, if $\gamma$ has positive increase, one may choose any $c>0$.
\end{proposition}

\section{Background on Diophantine Approximation}
    \label{appendix:diophantine}

In this appendix, we summarise some background information on Diophantine approximations that are used in the main body of the text. Recall the continued fraction expansion from Definition~\ref{def:continued_frac}.

\begin{lemma}[{See e.g.\ \cite[pg~45]{Baker_NT}}]\label{lem:cont_frac_recurrence}
The convergents ${p_n}/{q_n}=[a_0;a_1,\ldots,a_n]$ to a real number $\alpha = [a_0;a_1,a_2,\dots]$ are given by the recurrence $p_0=a_0$, $q_0=1$, $p_1=a_0 a_1+1$, $q_1=a_1$ and
\[
	p_n=a_n p_{n-1}+p_{n-2}, \quad
	q_n=a_n q_{n-1}+q_{n-2}, \quad \text{for } n\geq2.
\]
Moreover, they satisfy the formula
\[
    p_n q_{n+1} - p_{n+1} q_n
    = (-1)^{n+1}.
\]
\end{lemma}

\begin{lemma}[{See e.g.\ \cite[pg~47]{Baker_NT}}]\label{lem:cont_frac_bounds}
The convergents $p_n/q_n$ to $\alpha = [a_0; a_1, a_2, \ldots] \in \bbR$ satisfy
\[
	\frac{1}{(a_{n+1}+2)   q_n^2}
	< \modulus{ \alpha - \frac{p_n}{q_n} }
	< \frac{1}{a_{n+1}   q_n^2}.
\]
\end{lemma}

\begin{lemma}[{See e.g.\ \cite[pg~47]{Baker_NT}}]\label{lem:best_approximations}
Let $p_1/q_1, p_2/q_2, \ldots$ be the convergents to a real number $\alpha$.
Then for any integers $p, q$ with $0 < q < q_{n+1}$, we have
\[
    \modulus{q\alpha - p} \geq \modulus{q_n\alpha - p_n}.
\]
\end{lemma}

\begin{lemma}[{Legendre's theorem \cite[pg 27--29]{Legendre1798}}]\label{lem:Legendre}
For $\alpha \in\bbR$ and coprime integers $p,q$ with 
\[
	q>0\quad\text{and}\quad\modulus{ \alpha - \frac{p}{q} }
	< \frac{1}{2q^2},
\] 
we have $p/q$ is a convergent to $\alpha$.
\end{lemma}

\begin{lemma}[{Corollary of Khintchin's approximation theorem \cite[Satz II]{Khintchine1924}}]\label{lem:Khintchin}
Let $\eps >0$. Then for  Lebesgue a.e.~$\alpha \in \R$,  the inequality
\[
	\modulus{ \alpha - \frac{p}{q} }
	\leq 
	\frac{1}{q^2 (\log q)^{1+\eps}}
\]
has only finitely many solutions $(p,q) \in \Z^2$. 
\end{lemma}

It was proven in \cite[Theorem~II]{KuipersMeulenbeld1952} that there always exist quadratically good approximations $u/v$ even if we restrict $u$ and $v$ to a certain parity.

\begin{lemma}[{odd/odd approximations, see \cite[Theorem~II]{KuipersMeulenbeld1952}}]\label{lem:odd-odd-approx}
    For every irrational $\alpha \in \bbR$, there exist infinitely many rationals $u/v$ of the shape odd/odd such that
    \begin{equation}\label{eq:odd-odd-approx}
    	\modulus{ \alpha - \frac{u}{v} }
    	\leq 
    	\frac{1}{v^2}.
    \end{equation}
\end{lemma}
We make use of insights related to the latter result in Section~\ref{sec:posinc}. In fact, the results used are gathered in the following two lemmata.

\begin{lemma}\label{lem:oddapp0} 
For $n\in \N$, let $\tfrac{p_n}{q_n},\tfrac{p_{n+1}}{q_{n+1}}$ be consecutive convergents to an irrational $\alpha \in \bbR$. At least one number in each pair $(q_n, q_{n+1})$ and $(p_n, p_{n+1})$ must be odd.
\end{lemma}
\begin{proof}
    Recall that by Lemma~\ref{lem:cont_frac_recurrence} we have $q_{n+1}=a_{n+1} q_{n}+q_{n-1}$. If both $q_{n+1}$ and $q_{n}$ are even, $q_{n-1}$ is even as well. By induction,  all denominators $q_{n+1}, q_{n}, \ldots, q_1, q_0$ are even, which contradicts $q_0 = 1$.

    Similarly, by Lemma~\ref{lem:cont_frac_recurrence}, both $p_{n+1}$ and $p_n$ being even implies $p_{n+1}, p_{n}, \ldots, p_1, p_0$ are even. This contradicts that  $p_0 = a_0$ and $p_1 = a_0a_1 + 1$ cannot both be even.
\end{proof}

\begin{lemma}\label{lem:oddapp1} 
    Let $\alpha\in \bbR$ be irrational. Let $n\in \N$, $n\geq 2$, and $\tfrac{p_{n-1}}{q_{n-1}}, \tfrac{p_n}{q_n},\tfrac{p_{n+1}}{q_{n+1}}$ be three consecutive convergents to $\alpha$. Then there exist odd integers $u,v$ such that
    \[
      \modulus{\alpha-\frac{u}{v}}< \frac{2}{v^2}
    \] with $q_{n-1}\leq v\leq q_{n+1}$.     
\end{lemma}
\begin{proof}
    Recall from Lemma~\ref{lem:cont_frac_bounds} that for every convergent $p_n/q_n$, we have
    \[
        \modulus{\alpha-\frac{p_n}{q_n}} < \frac{1}{q_n^2}.
    \]
    Therefore, the claim follows if either $p_n/q_n$ or $p_{n+1}/q_{n+1}$ are of the shape odd/odd. By Lemma~\ref{lem:oddapp0}, it suffices to discuss the following cases:~either $p_n/q_n = \odd/\even$ and $p_{n+1}/q_{n+1} = \even/\odd$ or $p_n/q_n = \even/\odd$ and $p_{n+1}/q_{n+1} = \odd/\even$. 
    
    In either case, set $u= p_{n+1} - p_{n}$ and $v =q_{n+1} - q_{n}$. Then $u$ and $v$ are odd and, using Lemma~\ref{lem:cont_frac_bounds} and Lemma~\ref{lem:cont_frac_recurrence}, we get
    \begin{align*}
    	\left| \alpha - \frac{u}{v} \right|
    	& \leq \left| \alpha - \frac{p_{n+1}}{q_{n+1}} \right| +
    		\left|\frac{p_{n+1}}{q_{n+1}} - \frac{p_{n+1} - p_{n}}{q_{n+1} - q_{n}} \right| \\
    	& \leq \frac{1}{q_{n+1}^2}
    		+ \frac{\left| p_{n+1} (q_{n+1} - q_{n}) - q_{n+1} (p_{n+1} - p_{n}) \right|}{q_{n+1} (q_{n+1} - q_{n})}\\
    	& = \frac{1}{q_{n+1}^2} +
    			\frac{\left| p_{n} q_{n+1} - p_{n+1} q_{n}\right|}{q_{n+1} (q_{n+1} - q_{n})}\\
    	& = \frac{1}{q_{n+1}^2} + \frac{1}{q_{n+1} (q_{n+1} - q_{n})}
    	< \frac{2}{(q_{n+1}- q_{n})^2}
    	= \frac{2}{v^2}.
    \end{align*}
    The estimate $q_{n-1}\leq q_{n+1}-q_n$, which follows from Lemma~\ref{lem:cont_frac_recurrence}, and the various choices for $v$ establish the remaining estimates. 
\end{proof}

With the latter two results, we may also prove a variant of Lemma~\ref{lem:odd-odd-approx} as follows.

\begin{lemma}\label{lem:odd-odd-approx_ingrid}
Let $\alpha \in \bbR$ be irrational. Then there exist infinitely many pairs of odd integers $u,v$ such that
\begin{equation}\label{eq:odd-odd-approx_2}
	\left| \alpha - \frac{u}{v} \right|
	< 
	\frac{2}{v^2}.
\end{equation}
\end{lemma}
\begin{proof}
The claim follows from Lemma~\ref{lem:oddapp1} because $q_n\to \infty$ as $n\to\infty$.
\end{proof}

\end{document}